\documentclass[12pt]{article}         % Dokumentklasse

 \usepackage[T1]{fontenc}              % wichtig
 \usepackage[latin1]{inputenc}         % deutsche Umlaute
      \input{dehyphtex}                % Für deutsche Silbentrennung
 \usepackage[width=16cm, height=22cm]{geometry}   
 \usepackage{mathtools} 
 \usepackage{amsmath,amssymb}          % Mathesymbole
\usepackage{hyperref} 
 \usepackage[amsmath,thmmarks,hyperref]{ntheorem}

 \usepackage{icomma}                   % unterdrückt Leerstelle hinter Komma im Mathemodus (nötig für Dezimalzahlen
 \usepackage{enumerate}                % Für Aufzählungen
 \usepackage{color}                    % Für Farben
 \usepackage{setspace}                 % Gibt \setspace Kommando für variable Linebreaks
 \usepackage{needspace}                % Mit \needspace kann Platz reserviert werden.
 \usepackage{url}                      % ermöglich Hyperlinks
 \usepackage{mathrsfs}                 % Geschwungene Großbuchstaben, aufrufem mit \mathscr{.}
 \usepackage{graphicx}                 % Zum Einbinden von Grafiken 
 \usepackage{slashed}
 \usepackage{esint}

 \usepackage{xypic}

 \numberwithin{equation}{section}
 
 \usepackage[numbers,square]{natbib}
\theoremstyle{nonumberplain}  
\theoremheaderfont{\itshape}  
\theorembodyfont{\normalfont}  
\theoremseparator{.}  
\theoremsymbol{$\Box$}
\newtheorem{proof}{Proof} 
  
\theoremstyle{plain}  
\theoremheaderfont{\normalfont\bfseries}  
\theorembodyfont{\itshape}  
\theoremsymbol{~}
\theoremseparator{.}  
\newtheorem{proposition}{Proposition}[section]  
\newtheorem{corollary}[proposition]{Corollary}  
\newtheorem{lemma}[proposition]{Lemma}  
\newtheorem{theorem}[proposition]{Theorem}   

\theorembodyfont{\normalfont}  
 
\newtheorem{remark}[proposition]{Remark}
\newtheorem{example}[proposition]{Example}

\newtheorem{definition}[proposition]{Definition}

\theoremstyle{nonumberplain}
\theoremheaderfont{\normalfont\bfseries}  
\theorembodyfont{\itshape}  
\theoremsymbol{~}
\theoremseparator{.}

\newcommand*{\res}{\operatorname{res}}
\newcommand{\R}{\mathbb{R}}

\newcommand{\V}{\mathcal{V}}
\newcommand{\N}{\mathbb{N}}
\newcommand{\Z}{\mathbb{Z}}
\newcommand{\C}{\mathbb{C}}
\newcommand{\dd}{\mathrm{d}}

\newcommand{\End}{\mathrm{End}}

\newcommand{\id}{\mathrm{id}}

\renewcommand{\>}{\right\rangle}

\renewcommand{\tilde}{\widetilde}
\renewcommand{\hat}{\widehat}

\newcommand{\e}{\mathrm{e}}

\title{Asymptotic Expansions and Conformal Covariance of the Mass of Conformal Differential Operators}
\author{ Matthias Ludewig}

\begin{document}

\maketitle
 
\begin{center}
  Max-Planck Institut f\"ur Mathematik \\ 
  Vivatsgasse 7 / 53111 Bonn, Germany \\ \medskip
 maludewi@mpim-bonn.mpg.de \\ \medskip
\end{center}

\begin{abstract}
We give an explicit description of the full asymptotic expansion of the Schwartz kernel of the complex powers of $m$-Laplace type operators $L$ on compact Riemannian manifolds in terms of Riesz distributions. The constant term in this asymptotic expansion turns turns out to be given by the local zeta function of $L$. In particular, the constant term in the asymptotic expansion of the Green's function $L^{-1}$ is often called the mass of $L$, which (in case that $L$ is the Yamabe operator) is an important invariant, namely a positive multiple of the ADM mass of a certain asymptotically flat manifold constructed out of the given data. We show that for general conformally invariant $m$-Laplace operators $L$ (including the GJMS operators), this mass is a conformal invariant in the case that the dimension of $M$ is odd and that $\ker L = 0$, and we give a precise description of the failure of the conformal invariance in the case that these conditions are not satisfied.
\end{abstract}

\section{Introduction}

Let $(M, g)$ be a compact Riemannian manifold of dimension $n$ and let $L$ be a self-adjoint $m$-Laplace type operator, acting on sections of a metric vector bundle $\V$ over $M$. By this we mean that $L$ is a differential operator of order $2m$ such that the principal symbol of $L$ equals the principal symbol of $(\nabla^*\nabla)^m$ for some (hence any) connection on $M$. Examples of higher order operators which are of this type are the GJMS operators, which play a prominent role in conformal geometry (see Example~\ref{ExampleGJMS} below and the references given there).

For $s \in \C$, denote by $L^{-s}$ the complex powers of $L$, defined via functional calculus. If $\ker L \neq 0$, then we take these powers on the orthogonal complement of its kernel. We show in the case that $\frac{n}{2} - ms \notin \Z$, the Schwarz kernel of $L^{-s}$ has an asymptotic expansion near the diagonal in $\mathscr{D}^\prime(M \times M)$ of the form
\begin{equation} \label{ExpansionIntroduction}
  L^{-s}(x, y) ~\sim~ \sum_{j=0}^\infty \Phi_j^L(x, y) \,\binom{s-1+\frac{j}{m}}{\frac{j}{m}} I_{2ms+2j}(x, y),
\end{equation}
where $\Phi_j^L(x, y)$ are the heat kernel coefficients (coming from the Minakshisundaram-Pleijel short-time asymptotic expansion of the heat kernel) and the $I_\alpha$ are Riesz distributions, a certain one-parameter family of distributions depending meromorphically on $\alpha \in \C$. Unfortunately, the family $I_\alpha$ has simple poles at the values $\alpha = n +2k$ for $k \in \N_0$, so that the expansion \eqref{ExpansionIntroduction} does not make sense in the case that $\frac{n}{2} -ms \in \Z$. However, it turns out that in this case, we obtain an asymptotic expansion for $L^{-s}(x, y)$ if for each $j$ such that $2ms +2j = n+2k$, we take the finite part of $I_{\alpha}(x, y)$ at $\alpha = n+2k$ instead of $I_{2ms+2j}(x, y)$.
\medskip

The Riesz distributions $I_\alpha(x, y)$ are continuous functions near the diagonal if $\mathrm{Re}(\alpha) > n$ with $I_\alpha(x, x) = 0$. Hence we obtain that if in \eqref{ExpansionIntroduction}, we only take the sum up to $j= \lfloor \frac{n}{2} - ms \rfloor$, the difference of the two sides is a continuous function near the diagonal and can be evaluated there. We refer to this value as the {\em constant term} in the asymptotic expansion of $L^{-s}(x, y)$. It turns out that in the case that $\frac{n}{2} - ms \notin \Z$, the constant term is given by the local zeta function of $L$ at $x$,
\begin{equation} \label{ConstantTermIntroduction}
\left[ L^{-s}(x, y) -  \sum_{j=0}^{\lfloor \frac{n}{2} - ms \rfloor} \Phi_j^L(x, y) \, \binom{s-1+\frac{j}{m}}{\frac{j}{m}} I_{2ms+2j}(x, y)\right]_{y=x} =  \zeta_L(s, x)
\end{equation} 
Again, in the case that $\frac{n}{2} - ms \in \Z$, both sides have a simple pole at the relevant value, and the equality continues to hold if we take the finite part on both sides.

\medskip

For $s=1$ in the above, the constant term in the asymptotic expansion of the Green's function $L^{-1}(x, y)$ of $L$ at $x$ is called the {\em mass} of $L$ at $x$, denoted by $\mathfrak{m}(x, L)$ \cite{HermannHumbert}. By \eqref{ConstantTermIntroduction}, this constant term is a value of the local zeta function,
\begin{equation} \label{MassAndZetaIntroduction1}
  \mathfrak{m}(x, L) = \zeta_L(1, x)
\end{equation}
in the case that $n$ is odd or that $m > \frac{n}{2}$. If $n$ is even and $m \leq \frac{n}{2}$, the mass instead equals the finite part of the zeta function at $s=1$, 
\begin{equation} \label{MassAndZetaIntroduction2}
  \mathfrak{m}(x, L) = \mathrm{f.p.}_{s= 1} \zeta_L(s, x).
\end{equation}
The mass is particularly interesting for conformally covariant differential operators. One reason is that it is related to the positive mass conjecture, which is still unsolved to this day, to the author's knowledge (for a more detailed exposition, see Section~\ref{SectionConformalGeometry} below). 
Formulas \eqref{MassAndZetaIntroduction1} respectively \eqref{MassAndZetaIntroduction2} recognize $\mathfrak{m}(x, L_g)$ as a local zeta value and therefore potentially allow the use of techniques from spectral geometry to investigate its properties, as well as the theory of heat kernels, via the Mellin transform formula for the zeta function. 

\medskip

The main result of this article, however, is for conformally covariant differential operators, the mass is a pointwise conformal invariant in odd dimensions. Namely, if $L_g$ is a self-adjoint, conformally invariant $m$-Laplace type operator (which is made precise in Section~\ref{SectionConformalGeometry} below and includes the case that $L$ is the Yamabe operator, the Paneitz-Branson operator or, more generally a GJMS operator), then if $n$ is odd or $m > \frac{n}{2}$, the mass satisfies the transformation law
 \begin{equation*}
   \mathfrak{m}(x, L_h) = e^{(2m-n)f(x)}  \mathfrak{m}(x, L_g)
  \end{equation*}
if $h = e^{2f} g$ for a smooth function $f$ and $L_g$, $L_h$ are the operators corresponding to the metrics indicated by the subscript. In particular, $\mathfrak{m}(x, L_g) V_g^{p}$, $p= \frac{n-2m}{n}$, (where $V_g$ is the Riemannian volume density corresponding to $g$) defines a density of weight  $n-2m$, which depends only on the conformal class of $g$.

\medskip

Since $\mathfrak{m}(x, L) = \zeta_L(1, x)$ in the above situation, this result fits nicely into the previously known results about conformally invariant zeta values: It is well known that in even dimensions, the value of the {\em global} zeta function at zero, $\zeta_g(0)$ is a conformal invariant (in the sense that the value only depends on the conformal structure), while in odd dimensions, its derivative $\zeta^\prime_g(0)$ is conformally invariant, provided that $\ker L = 0$ \cite{ParkerRosenberg}, \cite{BransonOrsted2}. Furthermore, again in even dimensions, the residue of the local zeta function $\zeta_g(s, x)$ at $s=1$ is a {\em pointwise} conformal invariant, which transforms as a density of weight $2m-n$ under a conformal  change \cite{ParkerRosenberg}, \cite{PongeLogarithmicSingularity}. Thm.~\ref{ThmLocalInvariants} completes this picture, by stating that in odd dimensions (where $\zeta_g(s, x)$ has no pole at $s=1$), the value of the zeta function itself is conformally covariant, with the same transformation law as the residue in even dimensions. Notice also that in even dimensions, the invariants $\zeta_g(0)$ and $\res_{s=1}\zeta_g(s, x)$ are given in terms of local data, i.e.\ curvature and the coefficients of $L$, while the odd-dimensional invariants $\zeta_g^\prime(0)$ and $\zeta_g(1, x)$ are global invariants, depending on the geometry in a non-local fashion.

\medskip

The existence of the conformally covariant section $\mathfrak{m}(1, L_g) = \zeta_g(1, x)$ is a very strange phenomenon: It is a {\em pointwise} yet {\em global} invariant. For example, if $L$ is a GJMS operator, it is zero on odd-dimensional round spheres (see Thm.~\ref{ThmPositiveMassGJMS}), but if $x$ is a point in an odd-dimensional Riemannian manifold $(M, g)$ a neighborhood of which is {\em isometric} to a region in the round sphere $S^n$, the mass $\mathfrak{m}(x, L_g)$ still need not to be zero. And in fact, it often is not, as suitable versions of the positive mass theorem imply.

\medskip

In Section~\ref{SectionRiesz} and \ref{SectionHeatEquation}, we discuss the Riesz distributions $I_\alpha$ as well as the heat kernel of $\Delta^m$ in Euclidean space. In each case, we discuss how to transplant the relevant distributions to the manifolds. In Section~\ref{SectionMLaplaceZeta}, we introduce the relevant concepts of global analysis, to set notation: General $m$-Laplace type operators $L$ on compact Riemannian manifolds, as well as their complex powers, the Minakshisundaram-Pleijel asymptotic expansion of their heat kernel and the corresponding local zeta function. In Section~\ref{SectionExpansion}, we then show that the Schwartz kernel $L^{-s}$ has asymptotic expansions of the claimed form (see Thm.~\ref{ThmAsymptoticExpansion} and Thm.~\ref{ThmConstantTerm}). Then in Section~\ref{SectionConformalGeometry}, we turn to conformally covariant $m$-Laplace type operators and give a further discussion of Thm.~\ref{ThmLocalInvariants}, the positive mass conjecture and related questions. Finally, in Section~\ref{SectionProof}, we give a proof of the main theorem, Thm.~\ref{ThmLocalInvariants} by calculating the variation of $\zeta_g(1, x)$ under a conformal change (see Thm.~\ref{ThmMassInvariance}), which also yields information on the case that $n$ is even or that $\ker L \neq 0$, which shows that indeed, we do not obtain a conformal invariant in this case.

\medskip

\textbf{Acknowledgements.} I thank Bernd Ammann, Christian B\"ar, Andreas Hermann, and Andreas Juhl for helpful discussion, as well as Potsdam Graduate School, SFB 647 and the Max Planck Institute for Mathematics in Bonn for financial support.

\section{Riesz Distributions} \label{SectionRiesz}

The material on Riesz distributions reviewed in this section is classical and well-known, and we repeat it for convenience of the reader and to set notation. For proofs, refer to \cite[Ch.~I.1]{Landkof} or \cite[Ch.~V.1]{EliasStein}.

Let $V$ be an $n$-dimensional Euclidean vector space, $n\geq 2$. For $\alpha \in \C$ with $\mathrm{Re}(\alpha)>0$ and $\alpha \neq n+2m$, $m \in \N_0$, the {\em Riesz Potentials} are defined by
\begin{equation*}
  I_\alpha(v) := C(\alpha, n) |v|^{\alpha-n},~~~~~\text{where}~~~C(\alpha, n) := \frac{\Gamma\left(\frac{n-\alpha}{2}\right)}{2^\alpha \pi^{\frac{n}{2}} \Gamma\left(\frac{\alpha}{2}\right)}.
\end{equation*}
Note that the coefficient function $C(\alpha, n)$, viewed as a meromorphic function in $\alpha$, has simple poles at the numbers $\alpha = n+2m$, $m\in \N_0$, so that $I_\alpha$ also has simple poles there. Since for any $\alpha$ as above, the $I_\alpha$ are locally integrable functions, they can be considered as distributions on $V$, by setting
\begin{equation*}
  I_\alpha[\varphi] := \int_V I_\alpha(v) \varphi(v) \dd v
\end{equation*}
for test functions $\varphi \in \mathscr{D}(V)$.
 As such, they satisfy the recursion relation
\begin{equation} \label{RecursionRiesz}
  \Delta I_\alpha = I_{\alpha-2}
\end{equation}
whenever $\mathrm{Re}(\alpha)>2$, $\alpha \notin n+2\N_0$. This is easiest to verify by Fourier transform, once one calculates that for $\alpha < n$, we have $\mathscr{F}[{I}_\alpha](\xi) = |\xi|^{-\alpha}$ (here we used the convention for Fourier transform as e.g.\ in \cite{Shubin}). In fact, \eqref{RecursionRiesz}, which is valid only for $\alpha$ not equal to one of the singularities $n, n+2, \dots$, has the more general form
\begin{equation} \label{RecursionRiesz2}
  \Delta \mathrm{f.p.}I_\alpha = \mathrm{f.p.} I_{\alpha-2},
\end{equation}
where $\mathrm{f.p.}I_\alpha$ denotes the finite part of $I_\alpha$ at $\alpha$, i.e.\ the constant term in the Laurent expansion around the point $\alpha$ (which is just equal to $I_\alpha$ in case that $\alpha$ is not a pole).  Formula \eqref{RecursionRiesz2} is then valid for all $\alpha$ with $\mathrm{Re}(\alpha)>2$.

The relations \eqref{RecursionRiesz} respectively \eqref{RecursionRiesz2} allows to extend $I_\alpha$ to all of $\C$ as a meromorphic family of distributions (meaning that for any $\varphi \in \mathscr{D}(V)$, $I_\alpha[\varphi]$ is a meromorphic function defined on all of $\C$), by defining for parameters $\alpha$ with $\mathrm{Re}(\alpha) \leq 0$ and test functions $\varphi \in \mathscr{D}(V)$
\begin{equation} \label{ContinuationRiesz}
  I_\alpha[\varphi] := I_{\alpha+2k}[\Delta^k\varphi],
\end{equation}
where $k \in \N$ is any number such that $\mathrm{Re}(\alpha) +2k > 0$ and such that $\alpha + 2k \neq n+2m$ for some $m \in \N_0$. We will call these distributions {\em Riesz distributions}. Because of the recursion formula \eqref{RecursionRiesz}, this does not depend on the choice of $k$. We then have
\begin{equation} \label{RieszZero}
  I_0 = \delta_0,
\end{equation}
the delta distribution at zero, as again is easy to see from the Fourier transform. 

\begin{remark} \label{RemarkPolesOfI}
  Using the residue formulas for the Gamma function, one finds the residue of $I_{\alpha}$ at $\alpha = n + 2k$, $k \in \N_0$, to be
  \begin{equation} \label{PolesOfI}
    \mathrm{res}_{\alpha = n+2k} I_\alpha = \frac{(-1)^k}{k!2^{2k-1} (4\pi)^{\frac{n}{2}} \Gamma\left(\frac{n+2k}{2}\right)} |v|^{2k}.
  \end{equation}
  Note that this is a smooth function since one takes an even power of $|v|$.
\end{remark}

\begin{remark} \label{RemarkFinitePartOfI}
  The finite parts of $I_\alpha$ at the values $\alpha = n+2k$ involve logarithms. For example, we have
  \begin{equation} \label{FinitePartAtN}
    \mathrm{f.p.}_{\alpha = n} I_\alpha(v) = \frac{\psi\left(\frac{n}{2}\right) - \gamma +2 \log(2) - 2 \log(|v|)}{(4\pi)^{\frac{n}{2}} \Gamma\left(\frac{n}{2}\right)},
  \end{equation}
  where $\gamma$ is the Euler-Mascheroni constant and $\psi = {\Gamma^\prime}/{\Gamma}$ is the digamma function. 
 %To see this, notice that the poles are at most of order. Hence for $\varphi \in \mathscr{D}(\R^n)$, 
%  \begin{equation*}
%  \begin{aligned}
%    \mathrm{f.p.}_{\alpha = n} I_\alpha[\varphi] &= \frac{\dd}{\dd \alpha}\Bigr|_{\alpha=n} \left\{ (\alpha-n) I_\alpha[\varphi]\right\}
%    = \frac{\dd}{\dd \alpha}\Bigr|_{\alpha=n} \left\{ (\alpha-n) C(\alpha, n) \int_M |v|^{\alpha-n} \varphi(v) \dd v\right\}\\
%    &= \int_M\left(\frac{\dd}{\dd \alpha}\Bigr|_{\alpha=n} \bigl\{ (\alpha-n) C(\alpha, n)\bigr\} + \lim_{\alpha \rightarrow n} \bigl\{ (\alpha - n) C(\alpha, n) \bigr\}\log(|v|)\right) \varphi(v) \dd v
%  \end{aligned}
%  \end{equation*}
%  so that \eqref{FinitePartAtN} follows from Laurent expanding $C(\alpha, n)$ at $\alpha=n$.
\end{remark}

We now define the Riesz distributions for a compact Riemannian manifold $M$. These will be distributions on $M \bowtie M$, the set of points $(x, y) \in M \times M$ such that there exists a unique minimizing geodesic parametrized by $[0, 1]$ connecting $x$ to $y$. This is an open and dense set of full measure in $M \times M$. Let $U$ be the open set of vectors $(x, v) \in TM$ such that $(x, \exp_x(t v)) \in M \bowtie M$ for each $t \in [0, 1]$. Then the Riemannian exponential map $\exp: U \longrightarrow M \bowtie M$ is a diffeomorphism.

In each fiber $T_x M$ of the tangent bundle, we have the Riesz distributions introduced above. They fit together to give distributions $I^{TM}_\alpha$ on $TM$, which can be restricted to the open set $U$ to give distributions $I^U_\alpha \in \mathscr{D}^\prime(U)$. Now we define the Riesz distributions ${I}_\alpha = I^M_\alpha \in \mathscr{D}^\prime(M \bowtie M)$ by
\begin{equation} \label{DefinitionPushDownToM}
  {I}_\alpha[\varphi] := {I}^{U}_\alpha[ j \cdot \exp^*\varphi],
\end{equation}
for test functions $\varphi \in \mathscr{D}(M \bowtie M)$, where for $(x, v) \in U$, 
\begin{equation*}
j(x, v) = \det \bigl(d \exp_x|_v\bigr)
\end{equation*}
denotes the Jacobian determinant of the exponential map. As before, if $\mathrm{Re}(\alpha)>0$ and $\alpha \neq n+2k$, $k \in \N_0$,  $I_\alpha$ is a locally integrable function on $M \bowtie M$ and the Jacobian factor in \eqref{DefinitionPushDownToM} ensures that in this case,
\begin{equation}
  I_\alpha(x, y) = C(\alpha, n) \,d(x, y)^{\alpha-n},
\end{equation}
where $d(x, y)$ is the Riemannian distance function. 

\section{The $m$-Heat Equation} \label{SectionHeatEquation}

Again, let $V$ be an $n$-dimensional Euclidean vector space. For $m \in \N$, consider the $m$-heat equation 
\begin{equation*}
  \Bigl(\frac{\partial}{\partial t} + \Delta^m\Bigr) u(t, v) = 0,
\end{equation*}
where $\Delta$ is the Laplace operator on $V$. The corresponding fundamental solution $\e_t^m$ (defined as the Schwartz kernel of the heat semigroup $e^{-t\Delta^m}$) can be easily found using the Fourier transform, here denoted by $\mathscr{F}$; it is given by
\begin{equation} \label{FormulaEm}
  \e_t^m(v) = \mathscr{F}^{-1}[e^{-t|\xi|^{2m}}](v) =  (2 \pi)^{-n} \int_V e^{i\langle v, \xi\rangle - t |\xi|^{2m}} \dd \xi.
\end{equation}
In the same way as in Section~\ref{SectionRiesz} for the Riesz distributions, we define $\e_t^m(x, y)$ on the set $M \bowtie M$ of a Riemannian manifold $M$, which again makes sense since $\e_t^m(v)$ is spherically symmetric. In this case, $\e_t^m$ is a smooth function on $M \bowtie M$; we call $\e_t^m$ the {\em Euclidean $m$-heat kernel}.

If $m=1$, then the Fourier transform can easily be computed to be given by 
\begin{equation*}
\e_t^1(v) = (4\pi t)^{-\frac{n}{2}} e^{-\frac{|v|^2}{4t}}.
\end{equation*}
For larger $m$, no elementary formula is available. The value of $\e_t^m(v)$ at zero however, can be computed as follows.

\begin{lemma} \label{LemmaEmValueAtZero}
  For any $t>0$ and all $v \in V$, we have
  \begin{equation*}
    \e_t^m(v) \leq \e_t^m(0) = \frac{\Gamma\left(\frac{n}{2m}\right)}{m \,(4\pi)^{\frac{n}{2}} \Gamma\left(\frac{n}{2}\right)} t^{-\frac{n}{2m}}.
  \end{equation*}
\end{lemma}

\begin{proof}
Clearly
\begin{equation*}
  |\e_t^m(v)| \leq (2 \pi)^{-n} \int_V \bigl|e^{i\<v, \xi\> - t|\xi|^{2m}}\bigr| \dd \xi = (2 \pi)^{-n} \int_V e^{- t|\xi|^{2m}} \dd \xi = \e_t^m(0).
\end{equation*}
  Using the spherical symmetry of $\e_t^m(v)$, we obtain
  \begin{equation*}
    \e_t^v(0) = (2\pi)^{-n}\mathrm{vol}(S^{n-1}) \int_0^\infty e^{-t r^{2m}} r^{n-1} \dd r.
  \end{equation*}
  Substituting $u = t r^{2m}$ gives 
  \begin{equation*}
    \e_t^v(0) = t^{-\frac{n}{2m}} \frac{\mathrm{vol}(S^{n-1}) }{2m(2\pi)^n} \int_0^\infty e^{-u} u^{\frac{n}{2m}-1} \dd u.
  \end{equation*}
  Using the well-known formula $\mathrm{vol}(S^{n-1}) = 2 \pi^{\frac{n}{2}} / \Gamma\left(\frac{n}{2}\right)$ and the definition of the gamma function, we obtain the result.
\end{proof}

The following proposition connects $\e_t^m$ to the Riesz distributions defined above and is essential for what follows.

\begin{proposition} \label{PropositionIntegralE}
For $m \in \N$ and $\alpha \in \C$ with $\mathrm{Re}(\alpha)>0$, define
\begin{equation*}
  E_{m, \alpha}[\varphi] := \frac{1}{\Gamma(\alpha)} \int_0^1 (\e_t^m, \varphi)_{L^2} t^{\alpha-1} \dd t, ~~~~~~~ \varphi \in \mathscr{D}(V).
\end{equation*}
\begin{enumerate}[{\normalfont (1)}]
\item The function $\alpha \mapsto E_{m, \alpha}$ extends uniquely to an entire holomorphic function with values in $\mathscr{D}^\prime(V)$. 
\item The difference
\begin{equation*}
  \Psi_{m, \alpha} := E_{m, \alpha} - I_{2m\alpha}
\end{equation*}
is smooth for every $\alpha \in \C$, $\alpha \neq n +2k$, and $\mathrm{f.p}_{\alpha = n +2k}  \Psi_{m, \alpha}$ is smooth as well.
\item The function $\alpha \mapsto \Psi_{m, \alpha}$ is meromorphic on all of $\C$ as a family of distributions with values in $\mathscr{E}(V) = C^\infty(V)$, meaning that for every compactly supported distribution $\varphi \in \mathscr{E}^\prime(V)$, the complex-valued function $\alpha \mapsto \Psi_{m, \alpha}[\varphi]$ is meromorphic in the usual sense.
\item We have
\begin{equation} \label{FormulaFpQ}
  (\Psi_{m, \alpha})(0) = ( \Psi_{m, \alpha})[\delta_0] = \frac{\Gamma\left(\frac{n}{2m}\right)}{ (4\pi)^{n/2} \Gamma\left(\frac{n}{2}\right)\Gamma(\alpha)} \frac{1}{m\alpha - \frac{n}{2}}.
\end{equation}
as an equality of meromorphic functions.
\end{enumerate}
\end{proposition}

\begin{remark} \label{RemarkPolesOfPsi}
  Clearly, since $E_{m, \alpha}$ is holomorphic, $\Psi_{m, \alpha}$ has the same poles as $-I_{2m\alpha}$, where the poles of the latter are given in \eqref{PolesOfI}. Notice that by Remark~\eqref{RemarkPolesOfI}, the residues of $I_{2m\alpha}$ at $\alpha = n+2k$, $k \in \N_0$, are multiples of $|v|^{2k}$, hence they vanish when evaluated at zero if $k \geq 1$. This explains why $(\Psi_{m, \alpha})(0)$ has no poles at $n+2k$, $k \geq 1$, even though $\Psi_{m, \alpha}$ has.
\end{remark}

\begin{proof}
{\itshape Step 1.} We show that $E_{m, \alpha}$ extends to an entire family of distributions. To this end, we first verify that $E_{m, \alpha}$ is a well-defined distribution for each $\alpha$ with $\mathrm{Re}(\alpha)>0$. To this end, calculate for $\varphi \in \mathscr{D}(V)$
\begin{equation*}
  \bigl|(\e_t^m, \varphi)_{L^2}\bigr| =  \left| \int_V \mathscr{F}^{-1}[{\varphi}](\xi) e^{-t|\xi|^{2m}} \dd \xi \right| \leq \int_V \bigl|\mathscr{F}^{-1}[{\varphi}](\xi)\bigr| \dd \xi =  \bigl\|\mathscr{F}^{-1}[{\varphi}]\bigr\|_{L^1}.
\end{equation*}
Hence
\begin{equation*}
  \bigl| E_{m, \alpha}[\varphi]\bigr| \leq \int_0^1 \bigl|(\e_t^m, \varphi)_{L^2}\bigr| t^{\alpha-1} \dd t \leq  \bigl\|\mathscr{F}^{-1}[{\varphi}]\bigr\|_{L^1} \frac{1}{\alpha},
\end{equation*}
which shows that the $E_{m, \alpha}$ are well-defined distributions for $\mathrm{Re}(\alpha)>0$. To see that $E_{m, \alpha}$ extends meromorphically to all of $\C$, notice that for $\mathrm{Re}(\alpha) >1$, we have
\begin{equation*}
  E_{m, \alpha} [ \Delta^m \varphi] = - \frac{1}{\Gamma(\alpha)} \int_0^1 \bigl( {(\e_t^m)}^\prime, \varphi\bigr)_{L^2} t^{\alpha-1} \dd t = - \frac{1}{\Gamma(\alpha)} (\e_1^m, \varphi)_{L^2} + E_{m, \alpha-1}.
\end{equation*}
Hence the distributions $E_{m, \alpha}$ satisfy the recurrence relation
\begin{equation} \label{RecurrenceEalpha}
  E_{m, \alpha-1} = \Delta^m E_{m, \alpha} + \frac{1}{\Gamma(\alpha)} \e_1^m, 
\end{equation}
which can be used to holomorphically extend $E_{m, \alpha}$ to all of $\C$, as done with the Riesz distributions above.

{\itshape Step 2.} We show (2) and (3) for $0 < 2m \mathrm{Re}(\alpha) < n$. Then for $\varphi \in \mathscr{D}(V)$,
\begin{equation*}
\begin{aligned}
  \Gamma(\alpha) E_{m, \alpha}[\varphi] &= \int_0^1 \bigl(\mathscr{F}^{-1}[e^{-t|\xi|^{2m}}], \varphi\bigr)_{L^2}  t^{\alpha-1}\dd t\\
  &= \int_0^1 \bigl(e^{-t|\xi|^{2m}}, \mathscr{F}^{-1}[\varphi]\bigr)_{L^2}  t^{\alpha-1}\dd t\\
 &= \int_V \left(\int_0^1 e^{-t|\xi|^{2m}} t^{\alpha-1} \dd t\right) \mathscr{F}^{-1}[\varphi](\xi) \dd \xi\\
 &= \int_V \left(\int_0^{|\xi|^{2m}} e^{-u} u^{\alpha-1} \dd t\right) |\xi|^{-2m\alpha}\mathscr{F}^{-1}[\varphi](\xi) \dd \xi\\
 &= \int_V \bigl( \Gamma(\alpha) - \Gamma(\alpha, |\xi|^{2m})\bigr) |\xi|^{-2m\alpha} \mathscr{F}^{-1}[\varphi](\xi) \dd \xi\\
 &= \Gamma(\alpha) I_{2m\alpha}[\varphi] - \int_V \Gamma(\alpha, |\xi|^{2m}) |\xi|^{-2m\alpha} \mathscr{F}^{-1}[\varphi](\xi) \dd \xi,
\end{aligned}
\end{equation*}
where we used that $\mathscr{F}^{-1}[|\xi|^{-2m\alpha}] = I_{2m\alpha}$ and the definition
\begin{equation*}
  \Gamma(\alpha, x) := \int_x^\infty e^{-u} u^{\alpha-1} \dd u
\end{equation*}
of the incomplete gamma function. Note that the use of Fubini in the third step is justified because by the assumption on $\alpha$, the singularity of the integrand at $\xi=0$ is locally integrable. 

If $\chi \in \mathscr{D}(V)$ is a cutoff function with $\chi \equiv 1$ on $B_1(0)$ and $\chi \equiv 0$ on $V\setminus B_2(0)$, we have 
\begin{equation*}
 \Gamma(\alpha, |\xi|^{2m}) |\xi|^{-2m\alpha} = \chi(\xi) \Gamma(\alpha, |\xi|^{2m}) |\xi|^{-2m\alpha} + \bigl(1- \chi(\xi)\bigr) \Gamma(\alpha, |\xi|^{2m}) |\xi|^{-2m\alpha}
\end{equation*}
where the first summand is in $L^1(V) \cap \mathscr{E}^\prime(V)$ and the second summand is in $\mathscr{S}(V)$, the Schwartz space, because $\Gamma(\alpha, |\xi|^{2m})$ decays exponentially at infinity together with its derivatives (provided $\mathrm{Re}(\alpha) > 0$). Hence the Fourier transform
\begin{equation*}
  \Psi_{m, \alpha} = - \frac{1}{\Gamma(\alpha)} \mathscr{F}^{-1}\Bigl[\Gamma(\alpha, |\xi|^{2m}) |\xi|^{-2m\alpha}\Bigr]
\end{equation*}
is smooth. It is furthermore clear that $\Psi_{m, \alpha}$ is holomorphic in $\alpha$ as an $\mathscr{E}(V)$-valued function, as one can exchange differentiation and integration freely for $0 < 2m\mathrm{Re}(\alpha) < n$.

{\em Step 3.} We now extend these results to $\alpha \in \C$ with $\mathrm{Re}(\alpha)>0$ arbitrary. By the properties of the Riesz distributions, we have $\Delta^k (I_{2k} * \varphi) \equiv \varphi$ for any $\varphi \in \mathscr{D}(V)$, where $*$ denotes convolution. Let $\chi$ be the cutoff function from before and calculate
\begin{equation*}
\begin{aligned}
   \mathscr{F}^{-1}[\varphi] &= \mathscr{F}^{-1}\bigl[ \Delta^k\{I_{2k} * \varphi\}\bigr] \\
   &= \mathscr{F}^{-1}\bigl[ \Delta^k\bigl\{(\chi I_{2k}) * \varphi\bigr\}\bigr] + \mathscr{F}^{-1}\bigl[\Delta^k\bigl\{\bigl((1-\chi)I_{2k}\bigr)*\varphi\bigr\}\bigr]\\
   &= |\xi|^{2k}\mathscr{F}^{-1}\bigl[(\chi I_{2k}) *\varphi\bigr] - \mathscr{F}^{-1} \bigl[ \bigl( [\Delta^k, \chi] I_{2k} \bigr) * \varphi \bigr] + \mathscr{F}^{-1} \bigl[ \bigl( (1-\chi) (\Delta^k I_{2k})\bigr) * \varphi\bigr].
\end{aligned}
\end{equation*}
Notice that $\Delta^k I_{2k} = \delta_0$, which has support in zero, and since $\chi$ is supported on $V \setminus B_1(0)$, we have $(1-\chi)(\Delta^k I_{2k}) \equiv 0$, i.e.\ the third term is zero. Set
\begin{equation*}
  \Psi^{(k)}_{m, \alpha} := - \frac{1}{\Gamma(\alpha)} \mathscr{F}^{-1}\Bigl[\Gamma(\alpha, |\xi|^{2m}) |\xi|^{2k-2m\alpha}\Bigr].
\end{equation*}
In particular $\Psi^{(0)}_{m, k} = \Psi_{m, \alpha}$. Then for $\alpha \in \C$ with $0 < 2m\mathrm{Re}(\alpha) < n$, we obtain by the calculations from step 2 combined with the formula for $\mathscr{F}^{-1}[\varphi]$ that
\begin{equation} \label{SecondFormulaForPhi}
  \Psi_{m, \alpha} [\varphi] = \Psi_{m, \alpha}^{(k)} \bigl[(\chi I_{2k}) * \varphi \bigr] - \Psi_{m, \alpha}\bigl[ \bigl( [\Delta^k, \chi] I_{2k} \bigr) * \varphi\bigr]
\end{equation}
for all $\varphi \in \mathscr{D}(V)$. Since $\Psi_{m, \alpha}^{(k)}$ defines a holomorphic function on $0 < 2m\mathrm{Re}(\alpha) < n+2k$ with values in $\mathscr{E}(V) \subset \mathscr{D}^\prime(V)$, the formula \eqref{SecondFormulaForPhi} in fact is valid on this larger strip.

Now we argue that for $0 < 2m\mathrm{Re}(\alpha) < n +2k$, the formula \eqref{SecondFormulaForPhi} also defines a continuous functional on compactly supported {\em distributions} $\varphi \in \mathscr{E}^\prime(V)$. Indeed, the convolution map $C_1$ mapping $\varphi$  to $(\chi I_{2k}) * \varphi$ is a continuous map from $\mathscr{E}^\prime(V)$ to itself, while the convolution map $C_2$ mapping $\varphi$ to $([\Delta^k, \chi] I_{2k}) * \varphi$ is a continuous map from $\mathscr{E}^\prime(V)$ into $\mathscr{D}(V)$, as $[\Delta^k, \chi] I_{2k} \in \mathscr{D}(V)$ (for details on the convolution of distributions and proofs of these facts, see \cite{treves67}, Chapter~27). Hence we can write 
\begin{equation*}
  \Psi_{m, \alpha} = \Psi^{(k)}_{m, \alpha} \circ C_1 - (E_{m, \alpha} - I_{2m\alpha}) \circ C_2.
\end{equation*}
This then defines a continuous linear functional on $\mathscr{E}^\prime(V)$ which depends meromorphically on $\alpha$ for $0 < 2m \mathrm{Re}(\alpha) < n+2k$. Because continuous linear functionals on $\mathscr{E}^\prime(V)$ are exactly the smooth functions, this shows that $\Psi_{m, \alpha}$ is a smooth function for each $\alpha$ with $\mathrm{Re}(\alpha) >0$, depending meromorphically on $\alpha$ in this domain. Finally, by the recursion formulas \eqref{RecurrenceEalpha}, \eqref{RecursionRiesz2} for $E_{m, \alpha}$ and $I_{\alpha}$, we have
 \begin{equation} \label{RecursionPsi}
   \Psi_{m, \alpha-1} = \Delta^m \mathrm{f.p.}\Psi_{m, \alpha} + \frac{1}{\Gamma(\alpha)} \e_1^m.
 \end{equation}
This allows to extend the result to $\alpha$ with $\mathrm{Re}(\alpha) \leq 0$.

{\itshape Step 4.} It remains to show \eqref{FormulaFpQ}. Suppose first that $0 < \mathrm{Re}(\alpha) < \frac{n}{2m}$. A similar calculation as in step 2 then shows that
\begin{equation*}
  \frac{1}{\Gamma(\alpha)}\int_0^\infty (\e_t^m, \varphi)_{L^2} t^{\alpha-1} \dd t = I_{2m\alpha}[\varphi]
\end{equation*}
for all $\varphi \in \mathscr{D}(V)$, where by the assumption on $\alpha$, all appearing integrals are absolutely convergent. Hence
\begin{equation*}
\Psi_{m, \alpha}[\varphi] = E_{m, \alpha}[\varphi] - I_{2m\alpha}[\varphi] = -\frac{1}{\Gamma(\alpha)}\int_1^\infty (\e_t^m, \varphi)_{L^2}t^{\alpha-1} \dd t.
\end{equation*}
Let $\delta_\varepsilon$ be any sequence in $\mathscr{D}(V)$ with $\|\delta_\varepsilon\|_{L^1} = 1$ converging to $\delta_0$ (the delta distribution at zero) in the sense of distributions. Then by dominated convergence,
\begin{equation*}
\begin{aligned}
\Gamma(\alpha)\Psi_{m, \alpha}(0) &= -\lim_{\varepsilon\rightarrow 0} \int_1^\infty (\e_t^m, \delta_\varepsilon)_{L^2}t^{\alpha-1} \dd t = -\int_1^\infty \e_t^m(0)t^{\alpha-1} \dd t \\
&=  -\frac{\Gamma\left(\frac{n}{2m}\right)}{m \,(4\pi)^{n/2} \Gamma\left(\frac{n}{2}\right)} \int_1^\infty t^{\alpha-\frac{n}{2m}-1} = \frac{\Gamma\left(\frac{n}{2m}\right)}{m \,(4\pi)^{n/2} \Gamma\left(\frac{n}{2}\right)} \frac{1}{\alpha - \frac{n}{2m}}.
\end{aligned}
\end{equation*}
Here the integral is absolutely convergent by the assumption on $\alpha$. 

If now $ \mathrm{Re}(\alpha)>\frac{n}{2m}$, the distributions $\mathrm{f.p.} I_{2m\alpha}$ are continuous and satisfy $\mathrm{f.p.}I_{2m\alpha}(0) = 0$. Therefore,
\begin{equation*}
\begin{aligned}
  \mathrm{f.p.} \Psi_{m, \alpha}(0) &= E_{m, \alpha} - \underbrace{\mathrm{f.p.}I_{2m\alpha}(0)}_{=0} = \frac{1}{\Gamma(\alpha)}\int_0^1 e_t^m(0) t^{\alpha-1} \dd t \\
  &= \frac{\Gamma\left(\frac{n}{2m}\right)}{m \,(4\pi)^{n/2} \Gamma\left(\frac{n}{2}\right)\Gamma(\alpha)} \int_0^1 t^{\alpha-\frac{n}{2m}-1},
\end{aligned}
\end{equation*}
which again converges by the assumption on $\alpha$.
\end{proof}

\section{$m$-Laplace Type Operators and their Zeta Functions} \label{SectionMLaplaceZeta}

Let $(M, g)$ be a Riemannian manifold. For $m \in \N$, an $m$-Laplace type operator is a differential operator $L$ of order $2m$ acting on sections of some metric vector bundle $\V$ over $M$ such that $L$ has the principal symbol of the $m$-th power of a Laplace type operator. Equivalently, $L$ is an $m$-Laplace type operator if and only if $L - (\nabla^*\nabla)^m$ is a differential operator of order at most $2m-1$ for some (hence any) connection $\nabla$ on $\V$.

In this section, we will introduce the complex powers, the heat kernel and the zeta function of the operator $L$ and discuss the relations between these.

\medskip

Suppose that $M$ is compact and that $L$ is formally self-adjoint with respect to some fiber metric on $\V$. Then $L$ with domain $C^\infty(M, \V)$ is an elliptic, semibounded differential operator and as such, it has a unique extension to an unbounded self-adjoint operator on $L^2(M, \V)$. Also, its spectrum consists of eigenvalues of finite multiplicities and its eigenfunctions are smooth \cite[Lemma~1.6.3]{gilkey95}. Throughout, we will denote the eigenvalues by $\lambda_1 \leq \lambda_2 \leq \dots$ (repeated with multiplicity) and let $\varphi_1, \varphi_2, \dots$ a corresponding orthonormal basis of eigenfunctions.

\medskip

Since an $m$-Laplace type operator $L$ is elliptic, we can form its complex powers $L^{-s}$. Unless $L$ is positive, one has to specify here how to deal with the negative and zero spectral part. In this paper,  we adopt the convention that by definition, $L^{-s}$ is the operator whose Schwartz kernel is
\begin{equation*}
  L^{-s}(x, y) = e^{-i \pi s}\sum_{\lambda_j < 0} |\lambda_j|^{-s}\, \varphi_j(x)\otimes \varphi_j(y)^* + \sum_{\lambda_j > 0} \lambda_j^{-s} \,\phi_j(x) \otimes \phi_j(y)^*,
\end{equation*}
with the sum converging in the sense of distributions (notice that the first sum is finite while the second one is infinite). With this convention, $L^{-1}$ is the inverse of $L$ on the orthogonal complement of $\ker L$. For any $s \in \C$, $L^{-s}$ is a classical pseudo-differential operator of order $-2ms$ (c.f.~\cite[Thm.~11.2]{Shubin}) and hence its Schwarz kernel is a smooth function away from the diagonal (this is a fundamental property of pseudo-differential operators). Near the diagonal, it is not smooth, but its singularity can be quite explicitly described in form of an asymptotic expansion, as we will see below in Section~\ref{SectionExpansion}. 

\medskip

The asymptotic expansion of $L^{-s}(x, y)$ involves the heat kernel coefficients, which we discuss now. Let $p_t^L(x, y)$ be the heat kernel of $L$, i.e.\ the Schwartz kernel of the operator $e^{-tL}$ (which can e.g.\ be defined using spectral calculus). The heat kernel is smooth in all three variables on $M \times M \times (0, \infty)$ and is given in terms of the spectrum of $L$ by the formula
\begin{equation} \label{DefinitionHeatKernel}
  p_t^L(x, y) = \Pi(x, y) + \sum_{\lambda_j <0} e^{-t\lambda_j} \varphi_j(x)\otimes \varphi_j(y)^* + \sum_{\lambda_j >0} e^{-t\lambda_j} \,\phi_j(x) \otimes \phi_j(y)^*,
\end{equation}
where $\Pi(x, y)$ is the integral kernel of the projection onto the kernel of $L$ and the series on the right hand side converges in $C^\infty((0, \infty) \times M \times M, \V \boxtimes \V^*)$. It follows e.g.\ from the construction of \cite[Ch.~1]{GreinerHeatEquation} or \cite[Ch.~1.7]{gilkey95} that near the diagonal, $p_t^L(x, y)$ of $L$ has a short-time asymptotic expansion of the form
\begin{equation} \label{HeatKernelExpansion}
 p_t^L(x, y) ~\sim~ \e_t^m(x, y) \sum_{j=0}^\infty t^{\frac{j}{m}} \frac{\Phi_j^L(x, y)}{\Gamma\left(\frac{j}{m}+1\right)},
 \end{equation}
 where $\e_t^m$ is the Euclidean $m$-heat kernel considered in Section~\ref{SectionRiesz} and the $\Phi_j^L$ are certain smooth sections of $\V \boxtimes \V^*$ defined on $M \bowtie M$, the set  of pairs $(x, y) \in M \times M$ such that there is a unique minimizing geodesic between $x$ and $y$. These "correction terms" $\Phi_j^L(x, y)$ appearing in \eqref{HeatKernelExpansion} are locally computable quantities, which depend on the geometry of $M$ in a local fashion.
 Precisely, the asymptotic relation \eqref{HeatKernelExpansion} means that the difference
 \begin{equation} \label{HeatKernelExpansionRemainder}
  r_t^N(x, y) := p_t^L(x, y)- \e_t^m(x, y) \sum_{j=0}^N t^{\frac{j}{m}} \frac{\Phi_j^L(x, y)}{\Gamma\left(\frac{j}{m}+1\right)}
\end{equation}
satisfies the estimate
\begin{equation} \label{HeatRemainderEstimate}
  \bigl| \nabla_x^k \nabla_y^l r_t^N(x, y)\bigr| \leq C t^{\frac{N - k- l +1}{m} -\frac{n}{2m}}
\end{equation}
for all $t \in [0, T]$ and $m, l \in \N_0$, where the constant $C>0$ is uniform for $(x, y)$ in compact subsets of $M \bowtie M$ (see \cite{GreinerHeatEquation}, Lemma~1.44). 

\medskip

For $x \in M$, the {\em local zeta function} $\zeta_L(s, x)$ corresponding to $L$ is defined for $s \in \C$ with $\mathrm{Re}(s) > \frac{n}{2m}$ by the formula
\begin{equation} \label{LocalZeta}
  \zeta_L(s, x) := e^{-i\pi s} \sum_{\lambda_j <0}  |\lambda_j|^{-s}\varphi_j(x) \otimes \varphi_j(x)^* + \sum_{\lambda_j >0} \lambda_j^{-s}\varphi_j(x) \otimes \varphi_j(x)^* = L^{-s}(x, x).
\end{equation}
Applying the formula
\begin{equation*}
  \lambda^{-s} = \frac{1}{\Gamma(s)} \int_0^\infty t^{s-1} e^{-\lambda t} \dd t,
\end{equation*}
to the spectral formula for the heat kernel \eqref{DefinitionHeatKernel}, one obtains the integral representation
\begin{equation} \label{MellinTransform}
  \zeta_L(s, x) = \frac{1}{\Gamma(s)} \int_0^\infty t^{s-1} p_t^+(x, x) \dd V_g(x) + L_-^{-s}(x, x)
\end{equation}
for the zeta function, where
\begin{equation} \label{NegativeSpectralPartOfL}
   L^{-s}_-(x, y) = e^{-i \pi s} \sum_{\lambda_j < 0} |\lambda_j|^{-s}\varphi_j(x) \otimes \varphi_j(y)^*
\end{equation}
is the negative spectral part of $L^{-s}$ and
\begin{equation} \label{PositiveHeatKernel}
  p_t^+(x, y) :=  \sum_{\lambda_j>0}^\infty e^{-t\lambda_j} \,\phi_j(x) \otimes \phi_j(y)^*
\end{equation}
is the positive spectral part of the heat kernel
 (clearly, in the case that $L$ is a positive operator, we have $p_t^+ = p_t^L$). While the representations \eqref{LocalZeta} and \eqref{MellinTransform} are only valid for $s \in \C$ with $\mathrm{Re}(s) > \frac{n}{2m}$ (by Thm.~\ref{ThmAsymptoticExpansion}, for such values $s$, the integral kernel $L^{-s}(x, y)$ is continuous and hence can be evaluated on the diagonal) $\zeta_L(s, x)$ can be continued to a meromorphic function on the whole complex plane.

\begin{lemma}[Analytic Continuation of Zeta Function] 
  Let $N \in \N_0$. Then for $s \in \C$ with $m\mathrm{Re}(s) > \frac{n}{2} - N - 1$, the local zeta function \eqref{LocalZeta} is given by the formula
  \begin{equation} \label{ZetaFormula}
\begin{aligned}
  \zeta_L(s, x) &= \frac{1}{\Gamma(s)}\sum_{j=0}^N \frac{\Phi_j^L(x, x)\Gamma\left(\frac{n}{2m}\right)}{(4\pi)^{\frac{n}{2}} \Gamma\left(\frac{n}{2}\right)\Gamma\left(\frac{j}{m}+1\right) \left(ms+j -\frac{n}{2}\right)} \\
  &~~~~~ + \frac{1}{\Gamma(s)} \int_0^1 t^{s-1} r_t^N(x, x) \dd t - \frac{1}{\Gamma(s)} \int_0^1 t^{s-1} p_t^{\leq 0}(x, x) \dd t \\
  &~~~~~+ \frac{1}{\Gamma(s)} \int_1^\infty t^{s-1} p_t^+(x, x) \dd t + L_-^{-s}(x, x)
  \end{aligned}
\end{equation}
where $\Phi_j^L(x, x)$ are the heat kernel coefficients from \eqref{HeatKernelExpansion},  $r^N_t(x, x)$ is the remainder term from \eqref{HeatKernelExpansionRemainder}, and $p_t^{\leq 0}(x, x) := p_t^L(x, x) - p_t^+(x, x)$ is the non-positive spectral part of the heat kernel.
\end{lemma}

\begin{proof}
First notice that \eqref{ZetaFormula} indeed makes sense in the half plane claimed by the estimate \eqref{HeatRemainderEstimate} and the decay properties of $p_t^+(x, x)$ for $t \rightarrow \infty$.
Now split the integral in \eqref{MellinTransform} up into two pieces, one over $(0, 1]$ and one over $[1, \infty)$. In the first integral, write
\begin{equation*}
  p_t^+(x, x) = \e_t^m(x, x)\sum_{j=0}^N t^{\frac{j}{m}} \frac{\Phi_j^L(x, x)}{\Gamma\left(\frac{j}{m}+1\right)} + r_t^N(x, x) - p_t^{\leq 0}(x, x)
  \end{equation*}
Multiplying the individual terms by $t^{s-1}$, one integrates then over $(0, 1]$ to obtain \eqref{ZetaFormula}. In particular, the first term (the sum over $j$) can be obtained using the explicit formula for $\e_t^m(x, x)$ from Lemma~\ref{LemmaEmValueAtZero}, which yields
\begin{equation*}
  \int_0^1 t^{s-1} \left(\e_t^m(x, x)\sum_{j=0}^N t^{\frac{j}{m}} \frac{\Phi_j^L(x, x)}{\Gamma\left(\frac{j}{m}+1\right)} \right) \dd t = \sum_{j=0}^N \frac{\Phi_j^L(x, x) \Gamma\left(\frac{n}{2m}\right)}{m (4\pi)^{\frac{n}{2}} \Gamma\left(\frac{n}{2}\right) \Gamma\left(\frac{j}{m}+1\right) } \int_0^1 t^{s-1+\frac{j}{m} - \frac{n}{2m}} \dd t.
\end{equation*}
Carrying out the integral  gives the first term of \eqref{ZetaFormula}.
\end{proof}

\section{Asympotic Expansion of Complex Powers} \label{SectionExpansion}

In this section, we prove two different results regarding the asymptotic expansion of $L^{-s}$, Thm.~\ref{ThmAsymptoticExpansion} and Thm.~\ref{ThmConstantTerm}.

\begin{theorem}[Asymptotic expansion of $L^{-s}$] \label{ThmAsymptoticExpansion}
Let $L$ be a formally self-adjoint $m$-Laplace type operator, acting on sections of a metric vector bundle $\V$ over $M$.
Then for any $s \in \C$, the Schwartz kernel $L^{-s}(x, y)$ has a complete asymptotic expansion near the diagonal in terms of Riesz distributions. More precisely, we have
\begin{equation} \label{AsymptoticExpansion}
  L^{-s}(x, y)~ \sim~ \sum_{j=0}^\infty \Phi_j^L(x, y) \,\mathrm{f.p.}_{\alpha=s}\left\{\binom{\alpha-1+\frac{j}{m}}{\frac{j}{m}} I_{2m\alpha+2j}(x, y)\right\},
\end{equation}
meaning that the difference of $L^{-s}(x, y)$ and the first $N$ terms of the series is in $C^k(M \bowtie M, \V \boxtimes \V^*)$ whenever $N > n/2 + k - 1 - m\mathrm{Re}(s)$. Here the coefficients $\Phi_j^L(x, y)$ are the heat kernel coefficients appearing in \eqref{HeatKernelExpansion}.
\end{theorem}

Above, we denoted
\begin{equation*}
  \binom{\alpha}{\beta} = \frac{\Gamma(\alpha+1)}{\Gamma(\beta+1) \Gamma(\alpha-\beta+1)}.
\end{equation*}
If $\frac{n}{2}-ms \notin \Z$, then the terms of \eqref{AsymptoticExpansion} do not have poles at the relevant values and we obtain the asymptotic expansion
\begin{equation*}
  L^{-s}(x, y) ~\sim~ \sum_{j=0}^\infty \Phi_j^L(x, y) \,\binom{s-1+\frac{j}{m}}{\frac{j}{m}} I_{2ms+2j}(x, y).
\end{equation*}
If however $\frac{n}{2}-ms \in \Z$, then $ms+j - \frac{n}{2} \in  \N_0$ from $j = \frac{n}{2}-ms$ onward so that we obtain
\begin{equation*}
\begin{aligned}
 L^{-s}(x, y)~ \sim~ &\sum_{j=0}^{\frac{n}{2}-ms -1} \Phi_j^L(x, y) \binom{s-1+\frac{j}{m}}{\frac{j}{m}} I_{2ms+2j}(x, y) \\
 &~~~~~~~~~~~~+ \sum_{j=\frac{n}{2}-ms}^\infty \Phi_j^L(x, y) \,\mathrm{f.p.}_{\alpha=s}\left\{\binom{\alpha-1+\frac{j}{m}}{\frac{j}{m}} I_{2m\alpha+2j}(x, y)\right\}.
 \end{aligned}
\end{equation*}

\begin{example}
The case $s = 1$ is probably most relevant, since in this case, $L^{-s}$ is the Green's operator to $L$. In that case we have
\begin{equation*}
  L^{-1}(x, y) \sim \sum_{j=0}^\infty \Phi_j(x, y) \,\mathrm{f.p.}_{\alpha=1}I_{2m\alpha+2j}(x, y).
\end{equation*}
Writing out the right hand side explicitly in the case that $m=1$ (i.e.\ the case that $L$ is a Laplace type operator) gives two different answers depending on whether the dimension $n$ is even or odd. In the case that $n$ is odd, we have the asymptotic expansion
\begin{equation*}
   L^{-1}(x, y) ~\sim~ \sum_{j=0}^\infty \Phi_j(x, y) \frac{\Gamma\left(\frac{n}{2} -1-j\right)}{(4\pi)^{\frac{n}{2}} j!} \left(\frac{d(x, y)}{2}\right)^{2j+2-n}.
\end{equation*}
In the case that $n$ is even, the gamma function has a pole from the coefficient $j=\frac{n}{2}-1$ onward and we get
\begin{equation*}
\begin{aligned}
  L^{-1}(x, y) ~\sim~ \sum_{j=0}^{\frac{n}{2}-2}& \Phi_j(x, y) \frac{\left(\frac{n}{2} -2-j\right)!}{(4\pi)^{\frac{n}{2}} j!} \left(\frac{d(x, y)}{2}\right)^{2-n + 2j}\\
  & + \sum_{j=\frac{n}{2}-1}^\infty \Phi_j(x, y)  \,(-1)^{j-\frac{n}{2}}\frac{2\log\bigl(d(x, y)\bigr)}{(4\pi)^{\frac{n}{2}} \left(j+1-\frac{n}{2}\right)! j! }\left(\frac{d(x, y)}{2}\right)^{2-n+2j}.
\end{aligned}
\end{equation*}
\end{example}

\begin{proof}[of Thm.~\ref{ThmAsymptoticExpansion}]
Denote by $p_t^+(x, y)$ the positive spectral part of the heat kernel as in \eqref{PositiveHeatKernel}. Suppose first that $\mathrm{Re}(s)>\frac{n}{2m}$. Then  $L^{-s}(x, y)$ is continuous and we may integrate \eqref{DefinitionHeatKernel} termwise to obtain the pointwise equality
\begin{equation*}
  L^{-s}(x, y) = L_-^{-s}(x, y) + \frac{1}{\Gamma(s)} \int_0^\infty t^{s-1} p_t^+(x, y) \dd t.
\end{equation*}
Clearly, the first term of this sum is smooth on $M\times M$ and entire in $s$, so we only need to consider the second term. We split the integral up as
\begin{equation*}
\begin{aligned}
  \frac{1}{\Gamma(s)} \int_0^\infty t^{s-1} p_t^+(x, y) \dd t &= \frac{1}{\Gamma(s)} \int_0^1 t^{s-1} p_t^L(x, y) \dd t - \frac{1}{\Gamma(s)} \int_0^1 t^{s-1} p_t^{\leq 0}(x, y) \dd t  \\
  &~~~~+ \frac{1}{\Gamma(s)} \int_1^\infty t^{s-1} p_t^+(x, y) \dd t
\end{aligned}
\end{equation*}
where $p_t^{\leq 0}(x, y)$ denotes the non-positive spectral part of $p_t^L(x, y)$.
The last two terms are easily found to be smooth in $M \times M$ and entire in $s$. For the first term, we use the asymptotic expansion \eqref{HeatKernelExpansion} with some large $N \in \N$ (to be specified later) to obtain
\begin{equation*}
   \int_0^1 t^{s-1} p_t^L(x, y) \dd t = \sum_{j=0}^N \frac{\Phi^L_j(x, y)}{\Gamma\left(\frac{j}{m}+1\right)} \int_0^1 t^{s + \frac{j}{m}- 1}\e_t^m(x, y) \dd t + \int_0^1 t^{s-1} r_t^N(x, y) \dd t\\
\end{equation*}
By \eqref{HeatRemainderEstimate}, the last term is $C^k$ on $M \bowtie M$, provided $N > \frac{n}{2} + k - 1 - m\mathrm{Re}(s)$. That is, for any half plane $\{s \in \C \mid \mathrm{Re}(s) > -\ell \}$, we can make $R^N_3(s; x, y)$  holomorphic on this half plane with values in $C^k(M \bowtie M, \V \boxtimes \V^*)$, by choosing $N$ large enough. By Prop.~\ref{PropositionIntegralE}, we have
\begin{equation*}
\begin{aligned}
 \frac{1}{\Gamma(s)} \int_0^1 t^{s + \frac{j}{m}- 1}\e_t^m(x, y) \dd t &= \frac{\Gamma\left(s+\frac{j}{m}\right)}{\Gamma(s)}\Bigl( \mathrm{f.p.}_{\alpha=s}I_{2m\alpha + 2j}(x, y) + \mathrm{f.p.}_{\alpha=s}\Psi_{m, \alpha+\frac{j}{m}}(x, y)\Bigr)\\
 &= \mathrm{f.p.}_{\alpha = s} \left\{ \frac{\Gamma\left(\alpha+\frac{j}{m}\right)}{\Gamma(\alpha)}\Bigl( I_{2m\alpha + 2j}(x, y) + \Psi_{m, \alpha+\frac{j}{m}}(x, y)\Bigr)\right\}
 \end{aligned}
\end{equation*}
for $\mathrm{Re}(s) >0$ in the distributional sense and for $\mathrm{Re}(s) > \frac{n}{2m}$ pointwise, where the function $\mathrm{f.p.}_{\alpha = s} \Psi_{m, \alpha+\frac{j}{m}}(x, y)$ is smooth on $M \bowtie M$ (where we transferred $\Psi_{m, \alpha}$ from $V = T_x M$ to $M \bowtie M$ via the exponential map, in the same way as $I_\alpha$ and $\e_t^m$). Here, for the second equality, we used that the left hand side is holomorphic and that taking the finite part is linear.

Assembling all the pieces, we get that
\begin{equation} \label{RepresentationL}
\begin{aligned}
  L^{-s}(x, y) &= L_-^{-s}(x, y) - \frac{1}{\Gamma(s)} \int_0^1 t^{s-1} p_t^{\leq 0}(x, y) \dd t + \frac{1}{\Gamma(s)} \int_1^\infty t^{s-1} p_t^+(x, y) \dd t\\
  &~~~~+ \frac{1}{\Gamma(s)} \int_0^1 t^{s-1} r_t^N(x, y) \dd t\\
  &~~~~+ \sum_{j=0}^N \Phi_j^L(x, y) \mathrm{f.p.}_{\alpha=s}\left\{ \frac{\Gamma\left(\alpha+\frac{j}{m}\right)}{\Gamma(\alpha) \Gamma\left(\frac{j}{m} + 1\right)}  I_{2m\alpha+2j}(x, y)\right\} \\
  &~~~~+ \sum_{j=0}^N \Phi_j^L(x, y)\mathrm{f.p.}_{\alpha=s} \left\{ \frac{\Gamma\left(\alpha+\frac{j}{m}\right)}{\Gamma(\alpha) \Gamma\left(\frac{j}{m} + 1\right)}  \Psi_{m, \alpha+\frac{j}{m}}(x, y)\right\}.
\end{aligned}
\end{equation}
Since the residues as well as the finite part of $\Psi_{m, s+\frac{j}{m}}$ are smooth (see Remark~\ref{RemarkPolesOfPsi}), so is
\begin{equation*}
\begin{aligned}
&\mathrm{f.p.}_{\alpha=s} \left\{ \frac{\Gamma\left(\alpha+\frac{j}{m}\right)}{\Gamma(\alpha) \Gamma\left(\frac{j}{m} + 1\right)}  \Psi_{m, \alpha+\frac{j}{m}}(x, y)\right\} = \\
&\frac{\Gamma\left(s+\frac{j}{m}\right)}{\Gamma(s) \Gamma\left(\frac{j}{m} + 1\right)}  \mathrm{f.p.}_{\alpha=s} \Psi_{m, \alpha+\frac{j}{m}}(x, y)
+ \frac{\dd}{\dd \alpha}\Bigr|_{\alpha = s}  \left\{ \frac{\Gamma\left(\alpha+\frac{j}{m}\right)}{\Gamma(\alpha) \Gamma\left(\frac{j}{m} + 1\right)}\right\} \res_{\alpha = s} \Psi_{m, \alpha+\frac{j}{m}}(x, y).
\end{aligned}
\end{equation*}

Since the left hand side of \eqref{RepresentationL} is entire in $s$ as a holomorphic function with values in distributions on $M \bowtie M$, and the right hand side is holomorphic in $s$ for $m\mathrm{Re}(s) > \frac{n}{2} - 1 - N$ as a distribution-valued function by \eqref{HeatRemainderEstimate}, this representation can be made valid for any fixed $s \in \C$, by possibly increasing $N$. 
Furthermore, \eqref{HeatRemainderEstimate} implies that the term involving $r_t^N(x, y)$ is $C^k$ on $M \bowtie M$ if $N > \frac{n}{2} + k - 1 - m\mathrm{Re}(s)$, while all other terms on the right hand side except the Riesz distributions are smooth on $M \times M$. This finishes the proof.
\end{proof}

\begin{remark}
Assuming that $L$ has only positive eigenvalues, we have for $\mathrm{Re}(s) > \frac{n}{2m}$ the equality
  \begin{equation*}
  L^{-s}(x, y) = \frac{1}{\Gamma(s)} \int_0^\infty t^{s-1} p_t^L(x, y) \dd t.
\end{equation*}
This inequality continues to hold for  $0 < \mathrm{Re}(s) \leq \frac{n}{2m}$, but only in the {\em distributional sense}, i.e.\
\begin{equation*}
   (L^{-s} \varphi)(x) = \frac{1}{\Gamma(s)} \int_0^\infty t^{s-1} \int_M p_t^L(x, y) \varphi(y) \dd V_g(y) \dd t
\end{equation*}
for smooth functions $\varphi$. In particular, for $m=1$, $L^{-s}(x, y)$ is not necessarily positive for $0 < s \leq \frac{n}{2}$ even though formally, it is given as an integral over the positive function $p_t^L(x, y)$.
\end{remark}

Thm.~\ref{ThmAsymptoticExpansion} states that by adding more and more correction terms, the difference between $L^{-s}(x, y)$ and the asymptotic expansion becomes more and more regular. However, there is no pointwise control of the error term. In contrast, the next result gives a pointwise asymptotic expansion.

\begin{theorem}[The constant Term is a Zeta Function] \label{ThmConstantTerm}
  Near the diagonal, we have the asymptotic expansion
  \begin{equation*}
     L^{-s}(x, y) = \sum_{j=0}^{\lfloor \frac{n}{2}-ms\rfloor} \!\!\Phi^L_j(x, y) \,\mathrm{f.p.}_{\alpha = s} \left\{\binom{\alpha-1+\frac{j}{m}}{\frac{j}{m}} I_{2m\alpha+2j}(x, y)\right\} + \mathrm{f.p.}_{\alpha = s}\zeta_L(\alpha, x) + o(1),
  \end{equation*}
  where $\zeta_L(s, x)$ is the local zeta function of $L$, defined in \eqref{LocalZeta}. 
\end{theorem}

\begin{remark}
  Thm.~\ref{ThmConstantTerm} tells us that the constant term in the asymptotic expansion of $L^{-s}(x, y)$ is given by the value at $s$ of the local zeta function of $L$. In the case that $n - 2ms \in 2\Z$, the term in the asymptotic expansion with $j = \frac{n}{2}- ms$  is given by $\Phi^L_{\frac{n}{2}-ms}$ times
  \begin{equation*}
  \begin{aligned}
    \mathrm{f.p.}_{\alpha = 0} &\left\{ \binom{ \alpha - 1 +\frac{n}{2m}}{\frac{n}{2m} - s} I_{n+2m\alpha} \right\}\\
    &= \left.\frac{\dd}{\dd \alpha}\right|_{\alpha=0} \binom{\alpha - 1 + \frac{n}{2m}}{\frac{n}{2m} - s} \res_{\alpha=0} I_{n+2m\alpha} + \binom{\frac{n}{2m} - 1}{\frac{n}{2m} - s} \mathrm{f.p.}_{\alpha=0} I_{n+2m\alpha}\\
    &= \frac{1}{(4\pi)^{\frac{n}{2}}\Gamma\left(\frac{n}{2}\right)}\binom{\frac{n}{2m} - 1}{\frac{n}{2m} - s}\left( \frac{\psi(s)-\psi\left(\frac{n}{2m}\right)}{m} +\psi\left(\frac{n}{2}\right) - \gamma - 2 \log\left(\frac{d(x, y)}{2}\right)\right),
  \end{aligned}
  \end{equation*}
  where $\gamma$ is the Euler-Mascheroni constant and $\psi = {\Gamma^\prime}/{\Gamma}$ is the digamma function (compare Remark~\ref{RemarkFinitePartOfI}). Now since this term contains $\log$ terms, it is not exactly clear what should be considered to be the constant term of the asymptotic expansion, since any constant term can be absorbed into the logarithm. Thm.~\ref{ThmConstantTerm} gives a natural candidate for such a constant term. Notice that this ambiguity does not occur if we have $\Phi^L_{\frac{n}{2} - ms}(x, x) = 0$.
\end{remark}

\begin{proof}
By Prop.~\ref{PropositionIntegralE}, we have
\begin{equation*}
  \Psi_{m, \alpha+\frac{j}{m}}(x, x) = \frac{\Gamma\left(\frac{n}{2m}\right)}{ (4\pi)^{n/2} \Gamma\left(\frac{n}{2}\right)\Gamma\left(\alpha+\frac{j}{m}\right)} \frac{1}{m\alpha + j - \frac{n}{2}}
\end{equation*}
as meromorphic functions, where $\Psi_{m, \alpha+\frac{j}{m}}$
Plugging this into \eqref{RepresentationL}, we obtain
\begin{equation*}
\begin{aligned}
\sum_{j=0}^N \Phi_j^L(x, x) \mathrm{f.p.}_{\alpha=s}&\left\{\frac{\Gamma\left(\alpha+\frac{j}{m}\right)}{\Gamma(\alpha) \Gamma\left(\frac{j}{m} + 1\right)}  \Psi_{m, \alpha+\frac{j}{m}}(x, x)\right\}\\
&~~~~~~~= \sum_{j=0}^N \Phi_j^L(x, x) \frac{\Gamma\left(\frac{n}{2m}\right)}{\Gamma\left(\frac{n}{2}\right) \Gamma\left(\frac{j}{m} + 1\right)}   \mathrm{f.p.}_{\alpha=s}\left\{\frac{1}{\Gamma(\alpha)\left(m \alpha + j - \frac{n}{2}\right)}\right\}
\end{aligned}
\end{equation*}
Now comparing formula  \eqref{RepresentationL} with the formula \eqref{ZetaFormula} for the analytic continuation of the zeta function gives the result.
\end{proof}

 \section{Conformal Geometry and the Positive Mass Theorem} \label{SectionConformalGeometry}

In this section, we set the stage for conformal geometry. Let $M$ be a Riemannian manifold of dimension $n$. For $w\in \R$, let $|\Lambda|^{\frac{w}{n}}$ be the line bundle of densities of weight $w$ on $M$. It can be constructed as the vector bundle associated to the $\mathrm{GL}(n)$ frame bundle of $M$ via the one-dimensional representation $\rho_w(A) := |\det(A)|^{\frac{w}{n}}$.
In particular, $|\Lambda|^0 = \underline{\R}$ is the trivial real line bundle over $M$, while $|\Lambda|^1$ is the bundle of volume densities on $M$, i.e.\ the objects that can be integrated over. The choice of a metric $g$ on $M$ gives rise to global trivializations $i_{g, w}: \underline{\R}= |\Lambda|^0 \rightarrow |\Lambda|^{\frac{w}{n}}$, by mapping $1 \in \R$ to $V_g^{w/n}$, where $V_g$ is the Riemannian volume density associated to the Riemannian metric $g$.

From now on fix a conformal class $\mathcal{C} = \{e^{2f}g_0\mid f \in C^\infty(M)\}$ of metrics on $M$. If $\V$ is a metric (real or complex) vector bundle over $M$, we say that $L$ is a {\em conformally invariant $m$-Laplace type operator} with respect to the conformal class $\mathcal{C}$, acting on $\V$, if $L$ is a differential operator of order $2m$ mapping sections of $|\Lambda|^{\frac{n-2m}{2n}} \otimes \V$ to sections of $|\Lambda|^{\frac{n+2m}{2n}} \otimes \V$ such that for each metric $g \in \mathcal{C}$, the operator
\begin{equation} \label{DefinitionLg}
  L_g := i_{g, \frac{n+2m}{2}}^{-1} \circ L \circ i_{g, \frac{n-2m}{2}}
\end{equation}
is an $m$-Laplace type operator, acting on sections of $\V$ (here the $i_{g, w}$ are supposed to act trivally on the $\V$ factor). We say that $L$ is formally self-adjoint, if $L_g$ is formally self-adjoint for each metric $g \in \mathcal{C}$. Under the conformal change $h = e^{2f} g$, the corresponding operators $L_h$, $L_{g}$ transform according to the law
\begin{equation} \label{ConformalChangeOfL}
  L_h = e^{-\frac{n+2m}{2}f} L_{g} e^{\frac{n-2m}{2}f}.
\end{equation}
where $e^{-\frac{n+2m}{2}f}$ and $e^{\frac{n-2m}{2}f}$ are to be understood as multiplication operators.

\begin{remark}
  If $L$ is any differential operator mapping $|\Lambda|^{a}\otimes \V$ to $|\Lambda|^b \otimes \V$ such that $L_g = i_{g, b}^{-1} \circ L \circ i_{g, a}$ is a formally self-adjoint $m$-Laplace type operator for any metric $g$ in a conformal class $\mathcal{C}$, then we necessarily have $a= \frac{n-2m}{2}$ and $b= \frac{n+2m}{2}$. This follows easily using the restricted form of the principal symbol and self-adjointness. Hence the special weights considered above are no actual restriction.
\end{remark}

We now give several examples for conformally invariant $m$-Laplace type operators.

\begin{example}
If $\tilde{L}$ is an $m$-Laplace type operator with respect to a fixed metric $g$, then we obtain a conformally invariant $m$-Laplace type operator $L$ with respect to the conformal class $\mathcal{C} = \{e^{2f}g | f \in C^\infty(M)\}$ by setting
\begin{equation*}
  L_{h} = e^{-\frac{n+2m}{2}f} \tilde{L} e^{\frac{n-2m}{2}f}
\end{equation*}
for $h=e^{2f}g$. This defines $L$ using \eqref{DefinitionLg}. $L$ is formally selfadjoint if and only if $\tilde{L}$ is.
\end{example}

\begin{example}[The Yamabe Operator]
The most classical example of a conformally covariant differential operator is the {\em Yamabe operator}, which is a Laplace type operator ($m=1$) and given by the formula
\begin{equation*}
   L_gu = \Delta_gu + \frac{n-2}{4(n-1)} \mathrm{scal}_gu,
\end{equation*}
where $\mathrm{scal}_g$ is the scalar curvature of $M$. Given a connection $\nabla$ on $\V$, one can also define a twisted Yamabe operator acting on sections $\V$, which then is again a conformally invariant Laplace type operator (see \cite{ParkerRosenberg}, Section~1).
\end{example}

\begin{example}[The Paneitz-Branson Operator] \label{ExPaneitz}
An example of a conformally covariant $2$-Laplace type operator is the {\em Paneitz-Branson} operator, given by
\begin{equation*}
  L_gu = \Delta^2 + \delta\bigl((n-2) J_g - 4P_g)\bigr) d + \left( \frac{n}{2} - 2\right) \left(\frac{n}{2} J_g^2 - 2|P_g|^2 - \Delta J_g \right)
\end{equation*}
where 
\begin{equation*}
  J_g = \frac{\mathrm{scal}_g}{2(n-1)}, ~~~~~~~~ P_g = \frac{1}{n-2} (\mathrm{ric}_g - J_g g)
\end{equation*}
are the normalized scalar curvature and the Schouten tensor \cite{Paneitz} \cite{BransonDifferentialOperators}. 
\end{example}

\begin{example}[The GJMS Operators] \label{ExampleGJMS}
  On $n$-dimensional manifolds, there is a self-adjoint conformally covariant $m$-Laplace type operator $L_g$ acting on functions for any $1 \leq m \leq \frac{n}{2}$ if $n$ is even and arbitrary $m$ if $n$ is odd or conformally Einstein. These involve more and more complicated expressions in the curvature tensor and its derivatives, Example~\ref{ExPaneitz} being only the start. They were constructed by Graham, Jenne, Mason and Sparling in their seminal work \cite{GJMS1} and therefore called GJMS operators. The non-existence of GJMS operators with $m>\frac{n}{2}$ on general even-dimensional manifolds was shown in \cite{GoverHirachi}. The recursive structure of these operators was investigated by Juhl in \cite{Juhl1}, \cite{Juhl2},  \cite{Juhl3}, and \cite{Juhl4}. 
  
The GJMS operators are related to the problem of prescribing the $Q$-curvature of a Riemannian manifold in a fixed conformal class, just as the Yamabe-operator appears in the problem of describing the scalar curvature (i.e.\ the Yamabe problem). For this problem, see e.g.~\cite{ChangYangAnnals},\cite{DjaldiMacholdi}, \cite{Brendle} and the references therein.
  
For further references on the GJMS operators, see also \cite{BransonFunctionalDeterminant}, \cite{Gover}, \cite{BaumJuhl},  \cite{FeffermanGrahamJuhl}.
\end{example}

We are now ready to formulate the main result.

\begin{theorem}[Conformal Invariance of $\zeta_g(1, x)$]\label{ThmLocalInvariants}
Let $\V$ be a metric vector bundle over a compact $n$-dimensional Riemannian manifold $(M, g)$ and let $L$ be a conformally invariant, formally self-adjoint $m$-Laplace type operator with respect to a conformal structure $\mathcal{C}$ on $M$, acting on sections of $\V$. If $n$ is odd or $m > \frac{n}{2}$, and if $\ker(L_g) = 0$, then the value of the local zeta function of $L_g$ at one, $ \zeta_{g}(1, x) \in \mathrm{End}(\V_x)$, satisfies the transformation law
 \begin{equation*}
   \zeta_{h}(1, x) = e^{(2m-n)f(x)}  \zeta_{g}(1, x)
  \end{equation*}
if $h = e^{2f} g$ for $f \in C^\infty(M)$. That is, $\zeta_g(1, x)$ transforms as a density of weight $2m-n$ under a conformal change. 
\end{theorem}

The proof of this result will be carried out in the next section. Notice that if the dimension $n$ of $M$ is odd, then $\zeta_L(s, x)$ is regular at $s=1$ so that $\mathrm{f.p.}_{s=1} \zeta_g(s, x) = \zeta_g(1, x)$, while taking the finite part is really necessary if $n$ is even. 

\begin{definition}[Mass]
The {\em mass} of $L_g$ at a point $x \in M$ is defined by the formula
\begin{equation*}
  \mathfrak{m}(x, L_g) := \mathrm{f.p.}_{s=1} \zeta_g(s, x),
\end{equation*}
where $\zeta_g(s, x) := \zeta_{L_g}(s, x)$ is the local zeta function corresponding to $L_g$.
\end{definition}

Thm.~\ref{ThmLocalInvariants} tells us that for a conformally invariant $m$-Laplace type operator, the mass $\mathfrak{m}(x, L_g)$ is conformally covariant, i.e.\ it transforms as a density of weight $2m-n$ under a conformal change.

\begin{remark}
The notion of mass discussed in \cite{HermannHumbert}, \cite{AmmannHumbert} or \cite{HumbertRaulot} always makes the assumption that $M$ be flat near $x$. Our definition of mass works without this assumption, in virtue of Thm.~\ref{ThmConstantTerm}. 

Of course, if one restricts to conformal changes $f$ which are constant in a neighborhood of $x$, the conformal covariance of $\mathfrak{m}(x, L_g)$ from Thm.~\ref{ThmLocalInvariants} is trivial, by the conformal transformation law of the Green's function, which follows from \eqref{ConformalChangeOfL}. Proving Thm.~\ref{ThmLocalInvariants} for conformal changes $f$ not being constant near $x$ is the hard part.
\end{remark}

As mentioned in the introduction, in the case that $L$ is the Yamabe operator on $M$, $\mathfrak{m}(x, L_g)$ is related to the ADM-mass of the asymptotically flat manifold $(\overline{M}, \overline{g})$ built from $(M, g)$ in the following way: Suppose that the metric $g$ is flat in a neighborhood of the point $x \in M$ and set $\gamma(y) := 4(n-1)\mathrm{vol}(S^{n-1}) L^{-1}(x, y)$. Then the metric $\overline{g} := \gamma^{\frac{4}{n-2}}g$ is asymptotically flat on the manifold $\overline{M} = M \setminus \{x\}$ with zero scalar curvature and one can calculate its ADM mass $\mathfrak{m}_{\mathrm{ADM}}(\overline{M}, \overline{g})$, which is an invariant of asymptotically flat manifolds of relevance in general relativity. It turns out that that $\mathfrak{m}_{\mathrm{ADM}}(\overline{M}, \overline{g}) = C_n \mathfrak{m}(x, L_g)$, where $C_n$ is a positive constant depending only on the dimension. The {\em positive mass conjecture} then states that the ADM mass of an asymptotically flat manifold $(\overline{M}, \overline{g})$ of non-negative scalar curvature is always non-negative, and zero if and only if $(\overline{M}, \overline{g})$ is isometric to flat $\R^n$. The conjecture is known to be true in the case that $n \leq 7$ or that $M$ is locally conformally flat \cite{SchoenYauADM} \cite{SchoenYauADM2} \cite{SchoenVariational} \cite{SchoenYauConfFlat}, in the case that $M$ is spin \cite{WittenMass} \cite{ParkerTaubes} and in the case that $\mathrm{ric}_g \geq 0$ \cite[Prop.~10.2]{LeeParker}. 

The corresponding conjecture for compact manifolds is the following: If $(M, g)$ is a compact Riemannian manifold with Yamabe invariant $Y(M, g) > 0$ and $g$ is flat in a neighborhood of a point $x \in M$, then $\mathfrak{m}(x, L_g) \geq 0$, and $\mathfrak{m}(x, L_g) = 0$ if and only if $(M, g)$ is conformally diffeomorphic to the round sphere. It can be shown \cite[Prop.~4.1]{HermannHumbertOnPositiveMass} that the positive mass conjecture for compact manifolds and the one for asymptotically flat manifolds are equivalent, but solution of the general case seems not to be available. Recent progress has been made in \cite{HermannHumbert}, where the authors show using bordism techniques that if the positive mass conjecture is true for one oriented non-spin manifold of some dimension $n$, then it is true for all manifolds of that dimension. 

There is also a version of the positive mass theorem for a higher order operator, namely the Paneitz-Branson operator, which is a  conformally covariant $2$-Laplace operator (see Example~\ref{ExPaneitz} above): E.~Humbert and S.~Raulot show that the mass of the Paneitz-Branson operator on conformally flat manifolds is positive under suitable positivity conditions \cite{HumbertRaulot}. It seems natural to ask if there are analogs of positive mass theorems for higher order GJMS operators. A first start in this direction is maybe the following result.

\begin{theorem} \label{ThmPositiveMassGJMS}
Let $(M, g)$ be a conformally flat Riemannian manifold of dimension $n \geq 3$ with finite fundamental group and let $x \in M$. If $n$ is even, assume furthermore $M$ is flat near $x$. Let $L$ be a GJMS operator of order less than $n$. Then $m(x, L_g)$ is non-negative, and zero if and only if $m(x, L_g)$ is conformally equivalent to the standard sphere.
\end{theorem}

\begin{remark}
In the even-dimensional case, one indeed has to require that the metric be constant near the point $x$, because under a generic conformal change, one can achieve any real number as the mass at $x$, as one can see by Thm.~\ref{ThmMassInvariance} below. In odd dimensions of course, on can drop this assumption, by Thm.~\ref{ThmLocalInvariants}. 
\end{remark}

\begin{remark}
In fact, the assumptions imply that $(M, g)$ is the quotient of a sphere by a finite group of $O(n)$. This is the positivity assumption that replaces the requirement of non-negative scalar curvature in the standard positive mass theorem.
\end{remark}

\begin{proof}
By \cite[Thm.~2.8]{BransonFunctionalDeterminant}, \cite[Thm.~1.2]{Gover}, the GJMS operator $L$ of order $2m$ on  $S^n$ with its round metric $g_{S^n}$ is given  by the formula
\begin{equation*}
  L = \prod_{k=1}^m ( \Delta_{g_{S^n}} + c_k), ~~~~~~~~~~~~ c_k = \frac{1}{4} (n+2k-2)(n-2k).
\end{equation*}
Since $2m < n$ and $n \geq 3$, each of the values $c_k$, $k=0, \dots, m$ is positive. Hence $L$ is a positive operator. We claim that its Green's function equals
\begin{equation} \label{GreensFunctionOnSphere}
  G_{g_{S^n}}(x, y) = \frac{\Gamma \left( \frac{n-2m}{2}\right)}{(2 \pi)^{\frac{n}{2}} 2^m (m-1)!} \bigl(1 - \cos d(x, y)\bigr)^{\frac{n-2m}{2}}.
\end{equation}
Namely, remember from Section~\ref{SectionRiesz} that the Green's function in $\R^n$ is given by
\begin{equation*}
  L^{-1}_{g_{\R^n}}(x, y) =  \frac{\Gamma \left( \frac{n-2m}{2}\right)}{(4 \pi)^{\frac{n}{2}} (m-1)!} |x- y|^{2m-n}.
\end{equation*}
 We now write $S^n = \{(\cos r, v \sin r) \in \R \times \R^n \mid r \in \R, v \in S^{n-1}\}$, and without loss of generality let $x = (1, 0)$ be the north pole so that $r = d(x, y)$ is the spherical distance from $x$. Then the stereographic projection from the antipodal point $-x$ is given by 
  \begin{equation*}
     \sigma: S^n \setminus \{-x\} \longrightarrow M, ~~~~~~~~ (\cos r, v \sin r) \longmapsto \frac{\sin r}{1+\cos r} v,
  \end{equation*}
  and we have $\sigma^* g_{\R^n} = u^2 g_{S^n}$, where $u(y) = (1+\cos r)^{-1}$.  Hence, since $(\sin r)^2  =(1-\cos r)(1+\cos r)$, we have
  \begin{equation*}
  \begin{aligned}
  G_{g_{S^n}}(x, y) &= \frac{\Gamma \left( \frac{n-2m}{2}\right)}{(2 \pi)^{\frac{n}{2}} 2^m \Gamma(m)}  \bigl(1 + \cos r\bigr)^{\frac{n-2m}{2}} \left|\frac{\sin r}{1+\cos r}\right|^{2m-n} \\
  &=u(x)^{\frac{n+2m}{2}}u(y)^{\frac{n-2m}{2}} G_{\R^n}\bigl(\sigma(x), \sigma(y)\bigr).
  \end{aligned}
  \end{equation*}
  We therefore obtain from the transformation law of $L$, that $G_{g_{S^n}}$ as defined by \eqref{GreensFunctionOnSphere} satisfies $L_{g_{S^n}} G_{g_{S^n}} = \delta_x$. We conclude that it must be the Green's function of $L$, for if $L^{-1}_{g_{S^n}}$ is the Green's function, then $L_{g_{S^n}}(G_{g_{S^n}} - L^{-1}_{g_{S^n}}) = 0$, so $G_{g_{S^n}} - L^{-1}_{g_{S^n}}$ is smooth by elliptic regularity and hence in the kernel of $L_{g_{S^n}}$ as an unbounded operator on $L^2(S^n)$. But this implies that $G_{g_{S^n}} - L^{-1}_{g_{S^n}}$ is identically zero, because $L_{g_{S^n}}$ is a positive operator as noted above. This finishes the proof of the claim.

  Let $\hat{M}$ be the universal cover of $M$, with metric  $\hat{g} := \pi^* g$, where $\pi : \hat{M} \rightarrow M$ is the projection. Now it is well-known that any simply-connected conformally flat admits a conformal map into the standard sphere, by a monodromy argument (compare \cite[Section 1]{SchoenYauConfFlat}), hence we obtain a conformal map $\phi: \hat{M} \rightarrow S^n$. Since $\hat{M}$ is compact (as $\pi_1(M)$ is finite), this map is surjective, hence a covering map. Since $S^n$ does not admit non-trivial coverings for $n \geq 2$, we conclude that $\phi$ must be a diffeomorphism. 
  
  We obtain that $\hat{g} = |\det d\phi|^{\frac{2}{n} } g_{S^n}$ (as $\phi$ is a conformal map), and by the conformal transformation law of the GJMS operator $L$, we have that $L_{g}$ is a positive operator on $\hat{M}$ since it is conformally equivalent to $S^n$ and we get that its Green's function is given by
  \begin{equation*}
  \begin{aligned}
     L^{-1}_{\hat{g}}(\hat{x}, \hat{y}) &= |\det d \phi(\hat{x})|^{\frac{2m-n}{2n}}|\det d\phi(\hat{y})|^{\frac{2m+n}{2n}} L^{-1}_{g_{S^n}}\bigl(\phi(\hat{x}), \phi(\hat{y})\bigr)\\
     &= |\det d (\sigma \circ \phi)(\hat{x})|^{\frac{2m-n}{2n}}|\det d(\sigma \circ  \phi)(\hat{y})|^{\frac{2m+n}{2n}} L^{-1}_{g_{\R^n}}\bigl((\sigma \circ \phi)(\hat{x}), (\sigma \circ \phi)(\hat{y})\bigr).
  \end{aligned}
  \end{equation*}
  for points $\hat{x}, \hat{y} \in \hat{M}$. Now fix $x \in M$ and $\hat{x} \in \hat{M}$ with $\pi(\hat{x}) = x$. If $n$ is odd, then
  \begin{equation*}
    \mathfrak{m}(\hat{x}, L_{\hat{g}}) = |\det d \phi(x)|^{\frac{2m-n}{n}} \mathfrak{m}\bigl(\phi(\hat{x}), L_{g_{S^n}}\bigr)
  \end{equation*}
  by Thm.~\ref{ThmLocalInvariants}. However, $\mathfrak{m}(\phi(\hat{x}), L_{g_{S^n}}) = 0$ in odd dimensions, as easily seen from the explicit formula \eqref{GreensFunctionOnSphere}. If $n$ is even, then we made the additional assumption that $g$ is flat near $x$, which implies that also $\hat{g}$ is flat near $\hat{x}$. After choosing an isometry $\iota$ from a neighborhood $U$ of $\hat{x}$ to a neighborhood of zero in $\R^n$, we can therefore choose the conformal map $\phi$ in such a way that $\phi|_U = \sigma^{-1}\circ \iota$ (any local immersion into $S^n$ extends uniquely to a global one). Then for $\hat{y} \in U$, $(\sigma \circ \phi)(\hat{y}) = \iota(\hat{y}) = 0$ and
  \begin{equation*}
    |\det d(\sigma \circ  \phi)(\hat{y})| = |\det d \iota(\hat{y})| = 1,
  \end{equation*}
  as $\iota$ is an isometry. Therefore, the constant term in the asymptotic expansion at $x$ is equal to the constant term of the asymptotic expansion of the Green's function of $\Delta^m$ in $\R^n$, which is zero. This shows that $\mathfrak{m}(\hat{x}, L_{\hat{g}}) = 0$ in all dimensions. 
  
  Finally, as is easily checked, the Green's function of $g$ is given by
  \begin{equation*}
     G_g(x, y) = \sum_{\gamma \in \pi_1(M)} G_{\hat{g}}(\hat{x}, \gamma \cdot \hat{y}),
  \end{equation*}
  hence
  \begin{equation*}
     \mathfrak{m}(x, L_g) = \mathfrak{m}(\hat{x}, L_{\hat{g}}) + \sum_{\gamma \in \pi_1(M) \setminus \{1\}} G_{\hat{g}}(\hat{x}, \gamma \cdot \hat{x}) > 0.
  \end{equation*}
  since $\mathfrak{m}(\hat{x}, L_{\hat{g}}) = 0$ and $G_{\hat{g}}$ is positive. The proof is now finished.
  \end{proof}

Let us make some final observations. Let $L$ be a conformally invariant $m$-Laplace type operator, acting on functions on an odd-dimensional Riemannian manifold $M$. Then by Thm.~\ref{ThmLocalInvariants} and by the behavior of the volume form under a conformal change, the $(0, 2)$-tensor
\begin{equation*}
  g_{\mathrm{can}} := |\mathfrak{m}(x, L_g)|^{\frac{2}{2m-n}} g
\end{equation*}
is conformally invariant. If $\mathfrak{m}(x, L_g) \neq 0$ for all $x \in M$, then $g_{\mathrm{can}}$ is even a metric on $M$. For example, if $L$ is the Yamabe operator, this is the case if $n = 3, 5$ or if $M$ is conformally flat, by the positive mass theorem. For higher order GJMS operators, Thm.~\ref{ThmPositiveMassGJMS} gives some additional situations in which $\mathfrak{m}(x, L_g) \neq 0$ for all $x \in M$.

In case of the Yamabe operator, Habermann \cite{Habermann} calls $g_{\mathrm{can}}$ the {\em canonical metric} associated to the conformal structure. It is the unique metric in the conformal class that has  mass constant equal to one.

\section{The conformal variation of $\zeta_g(1, x)$} \label{SectionProof}

This section is dedicated to calculate the conformal variation of $\zeta_g(1, x)$ in the direction of a conformal change, a result which implies Thm.~\ref{ThmLocalInvariants}. Throughout, we will use the following notation. For a smooth function $F$ on $\mathcal{C}$ (with respect to its Fréchet topology) with values in some finite-dimensional vector space (such as $F(g) = \zeta_g(s, x)$ or $F(g) = p_t^{g}(x, y)$), we define the derivative at $g$ in direction of $f \in C^\infty(M)$ by
\begin{equation*}
  \delta_f F(g) := \left.\frac{\dd}{\dd \varepsilon}\right|_{\varepsilon=0} F\bigl(e^{2\varepsilon f} g\bigr).
\end{equation*}
We now prove the following more refined version of Thm.~\ref{ThmLocalInvariants}.

    \begin{theorem} \label{ThmMassInvariance}
     Let $M$ be a compact manifold of dimension $n$ and let $L$ be a conformally invariant $m$-Laplace type operator with respect to the conformal class $\mathcal{C}$.
     \begin{enumerate}
      \item If $n$ is odd or $m > \frac{n}{2}$, then we have
     \begin{equation*}
     \begin{aligned}
         \delta_f  \zeta_g(1, x) = (2m-n) f(x) \zeta_g(1, x) - 4m\bigl[ L_g^{-1} f \Pi_g\bigr](x, x)
     \end{aligned}
     \end{equation*}
     where $\Pi_g$ is the projection onto the kernel of $L_g$.
     \item If $n$ is even and $m \leq \frac{n}{2}$, we have
     \begin{equation*}
       \begin{aligned}
         \delta_f \bigl\{\mathrm{f.p.}_{s=1} \zeta_g(s, x)\bigr\} &= (2m-n)  f(x) \,\mathrm{f.p.}_{s=1} \zeta_g(s, x) - 4m\bigl[ L_g^{-1} f \Pi_g\bigr](x, x) \\
         &~~~~~~~~~~~~~~~~~~~~~~~~~~~~~~~~~~~~~~~~~~~~~~~+  2m{Q}_{\frac{n}{2}-m} f(x)
       \end{aligned}
     \end{equation*}
     where ${Q}_{\frac{n}{2}-m}$ is a certain differential operator of order $n-2m$ mapping functions to endomorphisms of $\V$, given in Lemma~\ref{LemmaAAsymptotics} below.
     \end{enumerate}
   \end{theorem}
   
Notice that the projection $\Pi_g$ onto the kernel  is a smoothing operator. Therefore $L_g^{-1} f \Pi_g$ is a smoothing operator as well and hence has a smooth integral kernel, which can be evaluated on the diagonal (remember that $L^{-1}$ is by our convention the inverse of $L_g$ {\em on the orthogonal complement of its kernel}). The differential operator $Q_{\frac{n}{2}-m}$ is implicitly defined in Lemma~\ref{LemmaAAsymptotics} below and is a differential operator which only depends on the local geometry of $M$ and the coefficients of $L$ near $x$.

  \begin{remark}
   In particular, this shows that if $n$ is even and the function $f$ has a zero of order $n-2m+1$ or greater at $x \in M$, then $\zeta_g(1, x)$ remaines unchanged under the conformal change $h = e^{2f} g$.
   \end{remark}
   
   \begin{remark}
     It is not hard to show that in the presence of a non-trivial kernel, $\zeta_g(1, x)$ is indeed not conformally covariant, even in odd dimensions. For example, if $L$ is the Yamabe operator on a flat odd-dimensional torus $M$, one can choose $f$ to be a suitable finite linear combination of sines and cosines, and then use the explicit spectral decomposition of $L$ together with trigonometric identites to show that  $[ L_g^{-1} f \Pi_g](x, x) \neq 0$ at most points $x$.
   \end{remark}
   
   For convenience of the reader, we give a sketch of the proof in the case that $L_g$ is a positive operator, which is considerably easier. Namely, in that case, we have for $\mathrm{Re}(s)$ large enough that
   \begin{equation*}
     \zeta_g(s, x)  = \frac{1}{\Gamma(s)}\int_0^\infty t^{s-1} p_t(x, x) \dd t
   \end{equation*}
   and the formula for the behavior of the heat kernel under an infinitesimal conformal change (Lemma~\ref{LemmaHeatKernelVariation} below) shows that one has
   \begin{equation*}
   \begin{aligned}
     \delta_f \zeta_g(s, x) &= - n f(x) \zeta_g(s, x) - \frac{2m}{\Gamma(s)} \int_0^\infty t^{s-1} \int_0^t \int_M \frac{\partial}{\partial t} p^g_{t-u}(x, y) f(y) p_u^{g}(y, x) \,\dd V_g(y)\, \dd u \dd t\\
     &= (2m - n)f(x) \zeta_g(s, x) \\
     &~~~~~~~~+ (s-1) \frac{2m}{\Gamma(s)} \int_0^\infty t^{s-2} \int_0^t \int_M p^g_{t-u}(x, y) f(y) p_u^{g}(y, x) \,\dd V_g(y)\, \dd u \dd t,
     \end{aligned}
   \end{equation*}
   where we integrated by parts with respect to $t$. Lemma~\ref{LemmaDerivativeBound} below ensures that the differentiation under the integral sign is justified. If we plug $s=1$, the second term will disappear, provided that it does not have a pole at this value of $s$. It turns out (Lemma~\ref{LemmaAAsymptotics}) that there is no pole if $n$ is odd, while a pole is present in the case that $n$ is even, with residue ${Q}_{\frac{n}{2}-m} f(x)$. This finishes our sketch of the proof. 
   
   We now first state and prove the necessary lemmas and then give a complete proof of Thm.~\ref{ThmMassInvariance}. The main difficulty remaining is how to deal with the case of a non-positive operator.
   
\begin{lemma} \label{LemmaHeatKernelVariation}
  Let $p_t^g(x, y)$ be the heat kernel of $L_g$. Then we have
  \begin{equation} \label{PointwiseVariationOfHeatKernel}
   \delta_f p_t^g(x, x) = - nf(x) p_t^g(x, x) - 2m\int_0^t \int_M \frac{\partial}{\partial t} p^g_{t-s}(x, y) f(y) p_s^{g}(y, x) \,\dd V_g(y)\, \dd s .
  \end{equation}
  for any $t>0$.
\end{lemma}

This result can be found for Laplace type operators in \cite[Lemma~1.1]{ParkerRosenberg} or \cite[Thm.~2.48]{bgv}. See also Prop.~3.5 in \cite{BransonOrsted}. Since the result seems not to be available in the literature in quite the generality and form given here, we give a full proof, for convenience of the reader.

\begin{proof}
Notice that by the formula \eqref{ConformalChangeOfL} for the conformal change of $L$, we have
\begin{equation} \label{VariationOfL}
\begin{aligned}
 \delta_f L_g &= \left. \frac{\partial}{\partial \varepsilon}\right|_{\varepsilon=0} \!\!\!\!\! e^{-\frac{n+2m}{2}\varepsilon f} L_g  e^{\frac{n-2m}{2}\varepsilon f} = - \frac{n+2m}{2} f  L_g + \frac{n-2m}{2} L_g  f\\
  &= -\frac{n}{2} [f, L_g] - m \{f, L_g\}
\end{aligned}
\end{equation}
in terms of commutator and anti-commutator.
  Let $e^{-tL_g}$ be the solution operator to the heat equation given by $L_g$. Differentiating the defining differential equation $\frac{\partial}{\partial t} e^{-tL_g} + L_g e^{-tL_g} = 0$, $e^{-tL_g}|_{t=0} = u$, gives
  \begin{equation*}
    \frac{\partial}{\partial t} (\delta_f e^{-tL_g}) + L_g (\delta_f e^{-tL_g}) = -(\delta_f L_g) e^{-tL}, ~~~~~~~ \delta_f e^{-tL_g}|_{t=0} = 0
  \end{equation*}
  Using \eqref{VariationOfL}, the Duhamel principle (i.e.\ uniqueness of solutions to the heat equation) implies therefore that $\delta_f e^{-tL_g}$ is given by
  \begin{equation} \label{DuhamelFormula}
    \delta_f e^{-tL} = \int_0^t e^{-(t-s)L_g}\left( \frac{n}{2} [f, L_g] + m \{f, L_g\} \right) e^{-sL_g}\dd s
  \end{equation}
  On the other hand, we have $V_{g_\varepsilon} = e^{n\varepsilon f} V_g$, hence
  \begin{equation*}
  \begin{aligned}
    (\delta_f e^{-tL_g})u(x) &= \left.\frac{\partial}{\partial \varepsilon}\right|_{\varepsilon=0} \int_M p_t^{g_\varepsilon}(x, y) u(y) \dd V_{g_\varepsilon}(y) \\
    &=  \int_M \delta_f p_t^g(x, y) u(y) \dd V_g(y) +  n \int_M p_t^g(x, y) u(y) f(y) \dd V_g(y).
    \end{aligned} 
    \end{equation*}
    Plugging this into \eqref{DuhamelFormula} yields the equality of kernels
    \begin{equation*}
      \delta_f p_t^g(x, y) = -n f(y) p_t^g(x, y) + \int_0^t \int_M p_{t-s}^g(x, z) \left( \frac{n}{2} [f, L_g]_z + m \{f, L_g\}_z \right)p_s^g(z, y) \dd V_g(z) \dd s.
    \end{equation*}
    If one evaluates this at the diagonal (i.e.\ sets $y = x$), then integrating by parts shows that the term involving $[f, L_g]$ does not contribute. Using furthermore that $p_t^g(x, y)$ satisfies the heat equation finishes the proof.
    \end{proof}

%    Formula \eqref{GlobalVariationOfHeatKernel} follows from \eqref{PointwiseVariationOfHeatKernel} by integration over $M$ and the Markhov property of the heat kernel. To prove the same formula for $p_t^+$ instead of $p_t^g$, let $\varphi_j$ be an eigenfunction of $L_g$ to the positive eigenvalue $\lambda_j$. Then by the above calculations,
%    \begin{equation*}
%      \begin{aligned}
%        (\varphi_j, \delta_f e^{-tL_g} \varphi_j)_{L^2} &= \frac{1}{2} \int_0^t \Bigl( e^{-(t-s)L_g} \varphi_j, \bigl( n [f, L_g] + 2m \{f, L_g\} \bigr) e^{-sL} \varphi_j \Bigr)_{L^2} \dd s\\
%        &= \frac{1}{2} \int_0^t e^{-(t-s)\lambda_j - s\lambda_j} \Bigl( \varphi_j, \bigl( n [f, L_g] + 2m \{f, L_g\} \bigr)\varphi_j \Bigr)_{L^2}\dd s\\
%        &= 2m t \lambda_j e^{-t\lambda_j} (\varphi_j, f\varphi_j)_{L^2} = -2m t \frac{\partial}{\partial t}  \bigl(e^{-tL_g} \varphi_j, f\varphi_j\bigr)_{L^2}.
%      \end{aligned}
%    \end{equation*}
%    Therefore 
%    \begin{equation*}
%    \begin{aligned}
%      \partial_f \int_M \tr\,p_t^+(x, x) \dd V_g(x)  &= \sum_{j=1}^\infty (\varphi_j, \delta_f  e^{-tL_g} \varphi_j)_{L^2} = -2mt \sum_{j=1}^\infty \frac{\partial}{\partial t}  \bigl(e^{-tL_g} \varphi_j, f\varphi_j\bigr)_{L^2},
%    \end{aligned}
%    \end{equation*}
%    which is \eqref{GlobalVariationOfHeatKernel}.

\begin{lemma} \label{LemmaDerivativeBound}
  For any $\varepsilon >0$ smaller than the smallest eigenvalue of $L_g$, there exists a constant $C>0$ such that
  \begin{equation} \label{ExponentialEstimate}
     \bigl| p_t^+(x, y) \bigr|, \bigl| \delta_f p_t^+(x, y)\bigr| \leq C e^{-t\varepsilon}
  \end{equation}
  for all $x, y \in M$ and $t\geq  1$.
\end{lemma}

\begin{proof}
For any $s \in \R$, the norm defined by
\begin{equation*}
   \|u\|_{H^s}^2 := \|L_g^{\frac{s}{2m}}\|_{L^2}^2 + \|\Pi_g u\|_{L^2}^2,
\end{equation*}
where $\Pi_g$ is the projection onto the kernel of $L_g$, is an equivalent $H^s$ Sobolev norm. Let $\Pi^+_g$ denote the orthogonal projection in $L^2(M, \V)$ onto the positive spectral part of $L_g$. Set $U_t := e^{-tL_g} \Pi^+_g$. Then $p_t^+(x, y)$ is the integral kernel of $U_t$. If now $u \in H^s(M, \V)$ has the decomposition $u = \sum_{j=1}^\infty u_j \varphi_j$ in terms of the eigendecomposition of $L$, we have
\begin{equation*}
  U_t u = \sum_{\lambda_j >0 } e^{-t\lambda_j} u_j \varphi_j
\end{equation*} 
and 
\begin{equation*}
   \|U_t u\|_{H^{s}}^2 = \sum_{\lambda_j>0} e^{-2 t \lambda_j} \lambda_j^{\frac{s}{m}} |u_k|^2 \leq \left( \sup_{\lambda_j>0} e^{-2t \lambda_j} \lambda_j^{\frac{s-r}{m}} \right) \|u\|_{H^r}^2 \leq C_{r, s}^2 e^{-2\varepsilon t} \|u\|_{H^r}^2,
\end{equation*}
for any $t \geq 1$, where one can choose $\varepsilon >0$ smaller than the smallest positive eigenvalue of $L$ freely and then
\begin{equation*}
  C_{r, s} := \sup_{\lambda_j >0} e^{-(\lambda_j-\varepsilon)} \lambda_j^{\frac{s-r}{2m}}.
\end{equation*}
It is well known for operators $A$ with a continuous integral kernel that the $\sup$ norm of the integral kernel equals the operator norm $\| A\|_{L^1, L^\infty}$. Since $L^1$ embeds continuously into $H^{-\frac{n+1}{2}}$, and $H^{\frac{n+1}{2}}$ embeds continuously into $L^\infty$, we have
\begin{equation*}
\bigl|p_t^+(x, y)\bigr| \leq \|U_t\|_{L^1, L^\infty} \leq B\|U_t\|_{H^{-\frac{n+1}{2}}, H^{\frac{n+1}{2}}} \leq B C_{-\frac{n+1}{2}, \frac{n+1}{2}} e^{-t\varepsilon}
\end{equation*}
for some constant $B>0$.

It remains to obtain the same result for $\delta_f p_t^+(x, y)$. Observe that since $L_g$ has only finitely many non-positive eigenvalues, the operator $\id - \Pi^+_g$ is a smoothing operator of finite rank, and we have $\delta_f \Pi^+_g = - \delta_f(\id - \Pi^+_g)$, so $\delta_f \Pi^+_g$ is also a smoothing operator of finite rank. Now using the Duhamel formula \eqref{DuhamelFormula} on $\delta_f U_t$, we obtain
\begin{equation} \label{DuhamelSecond}
\begin{aligned}
   \delta_fU_t &= \delta_f(\Pi^+_g e^{-tL} \Pi^+_g) = \Pi^+_g (\delta_f e^{-tL}) \Pi^+_g + U_t (\delta_f \Pi^+_g) + (\delta_f \Pi^+_g) U_t\\
   &= \int_0^t U_{t-s} \left( \frac{n}{2} [f, L_g] + m \{f, L_g\} \right) U_s\dd s + U_t (\delta_f \Pi^+_g) + (\delta_f \Pi^+_g) U_t,
\end{aligned}
\end{equation}
where we used several times that $\Pi^+_g$ is idempotent and commutes with $e^{-tL}$. Since $L_g$ is continuous from $H^s$ to $H^{s-2}$ for all $s \in \R$, we obtain
\begin{equation*}
\begin{aligned}
 \left\| \left(\frac{n}{2} [f, L_g] + m \{f, L_g\} \right) U_s\right\|_{H^r, H^s} &\leq D_{s} \|U_s\|_{H^s, H^{r+2}}&\leq D_{s}  C_{r+2, s} e^{-s\varepsilon}.
 \end{aligned}
\end{equation*}
Consequently, by \eqref{DuhamelSecond}, 
\begin{equation*}
\begin{aligned}
 \|\delta_f U_t\|_{H^r, H^s} &\leq \int_0^t \|U_{t-s}\|_{H^{s}, H^s} \left\|\left( \frac{n}{2} [f, L_g] + m \{f, L_g\} \right) U_s\right\|_{H^r, H^{s}}\dd s \\
 &~~~~~~~~~~~~~~~~~~~~~~~~~~~~~~~~+ C(r, s) \underbrace{\bigl(\|\delta_f \Pi^+_g\|_{H^r, H^r} + \|\delta_f \Pi^+_g\|_{H^s, H^s}\bigr)}_{:=E_{r, s}} e^{-t\varepsilon}\\
 &\leq D_s C_{s, s} C_{r+2, s}\int_0^t e^{-(t-s)\varepsilon} e^{-s\varepsilon} \dd s + E_{r, s} C_{r, s} e^{-t\varepsilon}= F_{r, s} e^{-t\varepsilon},
\end{aligned}
\end{equation*}
where $F_{r, s}>0$ is suitably chosen. Let $q_t(x, y)$ be the integral kernel of $\delta_f U_t$. Then
\begin{equation*}
 |q_t(x, y)| \leq \|\delta_f U_t\|_{L^1, L^\infty} \leq B \|\delta_f U_t\|_{H^{-\frac{n+1}{2}}, H^{\frac{n+1}{2}}} \leq B F_{-\frac{n+1}{2}, \frac{n+1}{2}}e^{-t\varepsilon}.
\end{equation*}
However, as seen in the proof of Lemma~\ref{LemmaHeatKernelVariation}, we have $q_t(x, y) = \delta_f p_t^+(x, y) + nf(y) p_t^+(x, y)$. The lemma now follows with the estimate on $p_t^+(x, y)$ from before.
\end{proof}

The following lemma is the key to the proof of Thm.~\ref{ThmMassInvariance} and the source for the term $Q_{\frac{n}{2}-m}f$ in the variation formula.

\begin{lemma} \label{LemmaAAsymptotics}
  Let $Q_t$ be the integral operator that sends functions $f \in C^\infty(M)$ to sections $Q_tf \in C^\infty(M, \End(\V))$ given by
    \begin{equation} \label{DefQt}
     (Q_tf)(x) := \frac{1}{t} \int_0^t \int_M p^g_{t-u}(x, y) f(y) p_u^{g}(y, x) \,\dd V_g(y)\, \dd u.
   \end{equation}
   Then  $(Q_t f)(x)$ has an asymptotic expansion as $t\rightarrow 0$, given by
   \begin{equation*}
      (Q_tf)(x) \sim \sum_{j=0}^\infty t^{\frac{j}{m}- \frac{n}{2m}} {Q}_j f(x),
   \end{equation*}
   where each $Q_j$ is a differential operator of degree $2j$ that takes functions to sections of $\End(\V)$. The coefficients of the operators ${Q_j}$ depend on the local geometry of $M$ and the coefficients of the operator $L$ near $x$. In particular, if $\V = \underline{\R}$, the trivial line bundle, then $Q_j$ is a $j$-Laplace type operator, up to a constant factor.
  \end{lemma}
  
  \begin{proof}
  The substitution $u= t\theta$ yields
  \begin{equation*}
    (Q_tf)(x) = \int_0^1 \int_M p^g_{t(1-\theta)}(x, y) f(y) p_{t\theta}^{g}(y, x) \,\dd V_g(y)\, \dd \theta.
   \end{equation*}  
   Because the heat kernel for $m$-Laplace type operators decays exponentially away from the diagonal (see for example \cite{GreinerHeatEquation} Lemma~1.2.4), the asymptotic expansion is the same when we replace $p_t^g(x, y)$ by $p_t^g(x, y) \chi(d(x, y))$ in the definition of $Q_t$, where $\chi: [0, \infty) \rightarrow [0, 1]$ is a smooth function with $\chi(r) = 1$ for $r \leq R/3$ and $\chi(r) = 0$ for $r \geq 2R/3$, where $R$ is the injectivity radius of $M$. Using the heat kernel asymptotic expansion \eqref{HeatKernelExpansion}, one then has for any large $N$ that up to order $\frac{N+1}{m} - \frac{n}{2m}$ in $t$, $(Q_tf)(x)$ has the same asymptotic expansion as
    \begin{equation*}
    \begin{aligned}
      \sum_{j+k \leq N} \frac{t^{\frac{j+k}{m}}}{\Gamma\left(\frac{j}{m}+1\right)\Gamma\left(\frac{k}{m}+1\right)} \int_0^1 (1-\theta)^j\theta^k \int_{M}\e_{t(1-\theta)}^m(x, y) \e_{t\theta}^m(x, y)  \tilde{\phi}_{jk}(y) \dd V_g(y) \, \dd \theta
      \end{aligned}
    \end{equation*} 
    where 
    \begin{equation*}
    \tilde{\phi}_{jk}(y) := \chi\bigl(d(x, y)\bigr)^2 \Phi_j^g(x, y) \Phi_k^g(y, x) f(y).
    \end{equation*} 
    We are left to calculate the asymptotic expansions of
    \begin{equation*}
      q_{jk}(t, x) :=  \int_0^1 (1-\theta)^j\theta^k \int_{M}\e_{t(1-\theta)}^m(x, y) \e_{t\theta}^m(x, y)  \tilde{\phi}_{jk}(y) \dd V_g(y) \, \dd \theta.
    \end{equation*}
    We now write this as an integral over the tangent space $V := T_x M$, by pulling the functions back using the Riemannian exponential map. Here we have to take into account the Jacobian factor $\det(g_{ij})^{1/2}$, where $g_{ij}$ are the components of the metric in normal coordinates around the point $x$. Setting ${\phi}_{jk} :=\det(g_{ij})^{1/2}\cdot  \exp_x^*\tilde{\phi}_{jk}$, we then obtain
        \begin{equation*}
    \begin{aligned}
    q_{jk}(t, x) &=  \int_0^1 (1-\theta)^j\theta^k \int_{V} \e^m_{t(1-\theta)}(v) \e^m_{t\theta}(v)  {\phi}_{jk}(v)\,\dd v\dd \theta\\
    &= (2\pi)^{-2n}\int_0^1 (1-\theta)^j\theta^k \int_{V}\int_{V}\int_{V} e^{i\langle v, \xi + \eta\rangle - t(1-\theta)|\xi|^{2m}-t\theta|\eta|^{2m}} {\phi}_{jk}(v)\,\dd \xi \dd \eta\dd v\dd \theta,
    \end{aligned}
    \end{equation*}
    where we used formula \eqref{FormulaEm} for $\e_t^m$. Substitution $\xi \mapsto t^{-\frac{1}{2m}}\xi$, $\eta \mapsto t^{-\frac{1}{2m}}\eta$, $v \mapsto t^{\frac{1}{2m}}v$ as well as applying Fubini's theorem gives
    \begin{equation*}
    \begin{aligned}
    \int_{V}\int_{V}\int_{V}&e^{i\langle v, \xi + \eta\rangle - t(1-\theta)|\xi|^{2m}-t\theta|\eta|^{2m}} {\phi}_{jk}(v)\,\dd \xi \dd \eta\dd v 
     \\
     &~~~~~~~~~~~~= t^{-\frac{n}{2m}}\int_{V}\int_{V}\int_{V} {\phi}_{jk}(t^{\frac{1}{2m}}v) e^{i\langle v, \xi + \eta\rangle - (1-\theta)|\xi|^{2m}-\theta|\eta|^{2m}}\dd \xi \dd \eta \dd v\\
     &~~~~~~~~~~~~= t^{-\frac{n}{2m}}\int_{V} {\phi}_{jk}(t^{\frac{1}{2m}}v) \mathscr{F}[{\gamma}^m_{1-\theta}](v) \mathscr{F}[{\gamma}^m_\theta](v) \dd v\\
     &~~~~~~~~~~~~= t^{-\frac{n}{2m}} \int_{V} {\phi}_{jk}(t^{\frac{1}{2m}}v) \mathscr{F}[\gamma^m_{1-\theta}*\gamma^m_\theta](v) \dd v,
     \end{aligned}
    \end{equation*}
    where $\mathscr{F}[{\gamma}_\theta]$ is the Fourier transform of $\gamma^m_\theta(\xi) := e^{-\theta |\xi|^{2m}}$ and $*$ denotes convolution. Clearly, $\mathscr{F}[\gamma^m_{1-\theta}*\gamma^m_\theta]$ is a Schwartz function on $V$. Now the function from $\R^+$ to the space of tempered distributions $\mathscr{S}^\prime(V)$ given by sending $t \in \R^+$ to the function $[ v \mapsto {\phi}_{jk}(t^{\frac{1}{2m}} v)]$ (considered as a distribution) has an asymptotic expansion in $\mathscr{S}^\prime(V)$ as $t \rightarrow 0$, given by its Taylor expansion around zero,
    \begin{equation*}
      \phi_{jk}(t^{\frac{1}{2m}} v) ~\sim~ \sum_{\alpha \in \N_0^n} t^{\frac{|\alpha|}{2m}} \frac{D^\alpha \phi_{jk}(0)}{\alpha!}   \, v^\alpha,
    \end{equation*}
    where we used the standard multi-index notation. We obtain the asymptotic expansion 
    \begin{equation*}
    \int_{V} \phi_{jk}(t^{\frac{1}{2m}}v) \mathscr{F}[{\gamma^m_{1-\theta}*\gamma^m_\theta}](v) \dd v ~\sim~ \sum_{\alpha \in \N_0^n}  t^{\frac{|\alpha|}{2m}} \frac{D^\alpha \phi_{jk}(0)}{\alpha!} \int_{V} v^\alpha \mathscr{F}[{\gamma^m_{1-\theta}*\gamma^m_\theta}](v) \dd v,
    \end{equation*}
    and standard manipulations give
    \begin{equation*}
    \begin{aligned}
      \int_{V} v^\alpha \mathscr{F}[{\gamma^m_{1-\theta}*\gamma^m_\theta}](v) \dd v &= (-i)^{|\alpha|}\int_{V} \mathscr{F}[{D^\alpha\gamma^m_{1-\theta}*\gamma^m_\theta}](v) \dd v \\
      &= (2 \pi)^{n}(-i)^{|\alpha|}\bigl(D^\alpha \gamma^m_{1-\theta} * \gamma^m_\theta\bigr)(0).
    \end{aligned}
    \end{equation*}
    On the other hand, as $\gamma^m_{1-\theta}*\gamma^m_\theta$ is an even function, its Fourier transform is real. Since $(D^\alpha \gamma^m_{1-\theta} * \gamma^m_\theta\bigr)(0)$ is real as well the above equality gives that only terms with $|\alpha|$ even can be non-zero.
%    Now it is easy to see that
%    \begin{equation*}
%      \bigl(D^\alpha \gamma^m_{1-\theta} * \gamma^m_\theta\bigr)(0) = \int_{V} e^{-|\xi|^{2m}} P_\alpha(\theta, \xi) \dd \xi = \int_0^\infty e^{-r^{2m}}\left( \int_{S^{n-1}}P_\alpha(\theta, r \xi) \dd \xi\right) \dd r
%    \end{equation*}
%    for some polynomial $P_\alpha$ in $\theta$ and $\xi$. Now it is easy to see that if one of the $\alpha_j$, $j=1, \dots, n$ is odd, then for each $\theta$, $P_\alpha(\theta, \xi)$ is an odd function in the variable $\xi_j$, so that the integral of $P_\alpha(\theta, r\xi)$ over the sphere is zero. We obtain that only multi-indices $\alpha = (\alpha_1, \dots, \alpha_n)$ such that $\alpha_j$ is even for each $j=1, \dots, n$ contribute to the asymptotic expansion.
    We obtain that $q_{jk}(t, x)$ has the asymptotic expansion
    \begin{equation*}
      q_{jk}(t, x) ~\sim~ (2 \pi)^{-n} \sum_{l=0}^\infty t^{\frac{l}{m} - \frac{n}{2m}} \sum_{|\alpha|=2l } \frac{(-i)^{|\alpha|}}{\alpha!} D^\alpha \phi_{jk}(0) \int_0^1 (1-\theta)^j \theta^k \bigl(D^\alpha \gamma^m_{1-\theta} * \gamma^m_\theta\bigr)(0)\, \dd \theta.
    \end{equation*}
    Putting everything together, we obtain that
    \begin{equation*}
     (Q_tf)(x) ~\sim~ \sum_{j, k, l = 0}^\infty t^{\frac{j+k+l}{m} - \frac{n}{2m}} \sum_{|\alpha| = 2l} c^\alpha_{jk} D_y^\alpha \bigl\{ \det\bigl(g_{ij}(y)\bigr)^{1/2}\Phi_j(x, y) \Phi_k(x, y) f(y)\bigr\}_{y=x},
    \end{equation*}
    where the constants $c^\alpha_{jk}$ are given by
    \begin{equation*}
      c^\alpha_{jk} := \frac{(-i)^{|\alpha|}}{\alpha! (2 \pi)^{n} \Gamma\left(\frac{j}{m}+1\right)\Gamma\left(\frac{k}{m}+1\right)} \int_0^1 (1-\theta)^j \theta^k \bigl(D^\alpha \gamma^m_{1-\theta} * \gamma^m_\theta\bigr)(0)\, \dd \theta.
    \end{equation*}
    Now $Q_if(x)$ is given by
    \begin{equation*}
      Q_i f(x) = \sum_{j, k, l = i} \sum_{|\alpha| = 2l} c^\alpha_{jk} D_y^\alpha \bigl\{ \det\bigl(g_{ij}(y)\bigr)^{1/2}\Phi_j(x, y) \Phi_k(x, y) f(y)\bigr\}_{y=x},
    \end{equation*}
    which is clearly a differential operator of order $2i$. The leading term is the one with $l=i$ and $j=k=0$, with all the derivatives falling onto $f$, i.e.\
    \begin{equation*}
    Q_i f(x) = \sum_{|\alpha| = 2i} c^\alpha_{00} \id_{\V_x} \cdot D_y^\alpha \bigl\{ f(y)\bigr\}_{y=x} + \text{lower order terms}.
    \end{equation*}
    However, since $Q_i$ is independent of the choice of normal coordinate system, the principal symbol of this operator at $x$ must be invariant under orthogonal transformations of $V = T_xM$. Since the principal symbol is a scalar multiple of $\id_{\V_x}$ and the only scalar $O(n)$-invariant homogeneous polynomial of order $2i$ is $|\xi|^{2i}$ (up to scaling), the principal symbol of $Q_i$ must be $\id_{\V_x} \cdot |\xi|^{2i}$.
\end{proof}

We are now in the position to prove Thm.~\ref{ThmMassInvariance}.

   \begin{proof}[of Thm.~\ref{ThmMassInvariance}]
   Notice first that by \eqref{MellinTransform}, we have for any $R>0$ and $s \in \C$ satisfying $\mathrm{Re}(s) > \frac{n}{2m}$
   \begin{equation} \label{ZetaFormulaR}
   \begin{aligned}
     \zeta_g(s, x) &= \frac{1}{\Gamma(s)} \int_0^R t^{s-1} p_t^g(x, x) \dd t - \frac{1}{\Gamma(s)} \int_0^R t^{s-1} p_t^{\leq 0}(x, x) \dd t \\
     &~~~+ \frac{1}{\Gamma(s)} \int_R^\infty t^{s-1} p_t^+(x, x) \dd t+ L_-^{-s}(x,x)
     \end{aligned}
   \end{equation}
   where $p_t^{\leq 0}(x, y) := p_t^g(x, y) - p_t^+(x, y)$ is the non-positive spectral part of the heat kernel and $L_-^{-s}(x, y)$ is the negative spectral part of $L^{-s}$ (see \eqref{NegativeSpectralPartOfL} and \eqref{PositiveHeatKernel}). 
   
    {\em Step 1.} We  differentiate the formula \eqref{ZetaFormulaR} in direction of a conformal change. By Lemma~\ref{LemmaDerivativeBound}, we may exchange differentiation and integration over $p_t^+(x, x)$. By Lemma~\ref{LemmaHeatKernelVariation}, we have
   \begin{equation*}
   \begin{aligned}
     &\int_0^R t^{s-1} \delta_f p^g_t(x, x) \dd t \\
     &= -2m \int_0^R t^{s-1} \int_0^t \int_M \frac{\partial}{\partial t} p^g_{t-u}(x, y) f(y) p^g_u(y, x) \dd V_g(y) \dd u \dd t - nf(x) \int_0^R t^{s-1}p^g_t(x, x) \dd t\\
     &= -2m R^s Q_R f(x) + 2m(s-1) \int_0^R t^{s-1} (Q_t f)(x) \dd t + (2m-n) f(x) \int_0^R t^{s-1} p^g_t(x, x) \dd t,
   \end{aligned}
   \end{equation*}
   where $Q_t$ is the operator from Lemma~\ref{LemmaAAsymptotics} and we integrated by parts in the second step.  Differentiating equation \eqref{ZetaFormulaR} and plugging this in, we get
   \begin{equation} \label{LargeFormula}
   \begin{aligned}
      \delta_f\zeta_g&(s, x) = (2m -n)f(x) \zeta_g(s, x) + \frac{2m}{\Gamma(s-1)} \int_0^R t^{s-1} (Q_t f)(x) \dd t- \frac{2m R^s}{\Gamma(s)} (Q_R f)(x) \\
      &+ (2m-n) f(x) \left[ \frac{1}{\Gamma(s)} \int_0^R t^{s-1} p_t^{\leq 0}(x, x) \dd t - \frac{1}{\Gamma(s)} \int_R^\infty t^{s-1} p_t^+(x, x) \dd t - L^{-s}_-(x, x) \right]\\
      &- \frac{1}{\Gamma(s)} \int_0^R t^{s-1} \delta_f p_t^{\leq 0}(x, x) \dd t + \frac{1}{\Gamma(s)} \int_R^\infty t^{s-1} \delta_f p_t^+(x, x) \dd t + \delta_f L_-^{-s}(x, x).
    \end{aligned}
   \end{equation}
   We now proceed to evaluate this formula at $s=1$, keeping in mind that if the dimension $n$ is even and $m \leq \frac{n}{2}$, we cannot evaluate  directly but instead need to take the finite part at $s=1$  on both sides. Here the integral over $Q_tf$ needs to be considered separately. 
   By Lemma \eqref{LemmaAAsymptotics}, we can write for $N \in \N$ large 
   \begin{equation*}
   (Q_t f)(x) = \sum_{j=0}^N t^{\frac{j}{m}-\frac{n}{2m}} ({Q}_j f)(x) + R^N_t(x)
   \end{equation*}
  for some remainder term $R_t^N(x)$ satisfying $|R^N_t (x)| \leq t^{\frac{N+1}{m} -\frac{n}{2m}}$. Therefore, we get
   \begin{equation}\label{IntegralOverQ}
     \begin{aligned}
       \int_0^R t^{s-1} (Q_tf)(x) \dd t &=\sum_{j=0}^N ({Q}_j f)(x) \int_0^R  t^{s-1+\frac{j}{m}-\frac{n}{2m}} \dd t  + \int_0^R t^{s-1} R_t^N(x)\dd t \\
       &= \sum_{j=0}^N({Q}_j f)(x)  \frac{R^{s+\frac{j}{m}-\frac{n}{2m}}}{s -\frac{n}{2m} + \frac{j}{2m}} +  \int_0^R t^{s-1} R_t^N(x)\dd t
     \end{aligned}
     \end{equation}
     for all $s \in \C$ with $\mathrm{Re}(s)$ large enough, depending on $N$.
    By choosing $N > \frac{n}{2}-m$, we can make \eqref{IntegralOverQ} valid in a neighborhood of $s=1$, by the estimate on the remainder term $R_t^N(x)$. We now need to divide by $\Gamma(s-1)$ and evaluate at $s=1$. Notice however, that the function $\Gamma(s-1)$ has a simple pole at $s=1$ with reside one, so its inverse has a simple zero with derivative one. Hence if \eqref{IntegralOverQ} is regular at $s=1$, we obtain zero. This is the case if $n$ is odd or if $m > \frac{n}{2}$, while if $n$ is even with $m \leq \frac{n}{2}$, the integral over $Q_t$ has a simple pole at one coming from the term with $j=\frac{n}{2} - m$. We obtain
     \begin{equation*}
        \left.\frac{2m}{\Gamma(s-1)} \int_0^R t^{s-1} (Q_tf)(x) \dd t \right|_{s=1} = \begin{cases} 0 & \text{if}~n~\text{is odd or}~m > \frac{n}{2}\\
          2m {Q}_{\frac{n}{2}-m} f(x) & \text{otherwise.}
          \end{cases}
     \end{equation*}
     All other terms of \eqref{LargeFormula} can be directly evaluated at one, so if $n$ is odd or $m > \frac{n}{2}$, we obtain  that
     \begin{equation*}
     \begin{aligned}
        \delta_f\zeta_g(1, x) &= (2m-n)f(x) \zeta_g(1, x) - 2m R (Q_R f)(x)\\
        &~~~~~~+ (2m-n) f(x) \left[ \int_0^R p_t^{\leq 0}(x, x) \dd t -\int_R^\infty p_t^+(x, x) \dd t - L^{-1}_-(x, x) \right]\\
      &~~~~~~-\int_0^R\delta_f p_t^{\leq 0}(x, x) \dd t + \int_R^\infty \delta_f p_t^+(x, x) \dd t + \delta_f L_-^{-1}(x, x)
     \end{aligned}
     \end{equation*}  
     while if $n$ is even and $m \leq \frac{n}{2}$, we get
     \begin{equation*}
     \begin{aligned}
        \delta_f\bigl\{\mathrm{f.p.}_{s=1}\zeta_g(s, x)\bigr\} &= (2m-n)f(x) \mathrm{f.p.}_{s=1}\zeta_g(s, x) +2m ({Q}_{\frac{n}{2}-m}f)(x)- 2m R (Q_R f)(x)\\
        &~~~~~~+ (2m-n) f(x) \left[ \int_0^R p_t^{\leq 0}(x, x) \dd t -\int_R^\infty p_t^+(x, x) \dd t - L^{-1}_-(x, x) \right]\\
      &~~~~~~-\int_0^R\delta_f p_t^{\leq 0}(x, x) \dd t + \int_R^\infty \delta_f p_t^+(x, x) \dd t + \delta_f L_-^{-1}(x, x).
     \end{aligned}
     \end{equation*}
     {\em Step 2.} Since these formulas are valid for all $R>0$, the plan is now to take the limit as $R \rightarrow \infty$. Notice that if $L$ is a positive operator, then one can directly see that all the terms involving $R$ tend to zero as $R \rightarrow \infty$ and we are left with the variation formulas claimed. In particular, we have
\begin{equation*}
   \int_R^\infty p_t^+(x, x) \dd t, ~~\int_R^\infty \delta_f p_t^+(x, x) \dd t ~~~\longrightarrow 0
\end{equation*}     
     as $R \rightarrow 0$, the first term by decay properties of $p_t^+$ and the second by Lemma~\ref{LemmaDerivativeBound}.
     
      In the general case, we have to invest some more work. Namely, we will see that each individual summand has an asymptotic expansion as $R \rightarrow \infty$, where the unbounded terms consist of a finite linear combination of $R$, $e^{-R\lambda_j}$ and $Re^{-R\lambda_j}$, with $\lambda_j$ running over the negative eigenvalues of $L_g$. Since we a priori know that the term converges, we get that the sum of the coefficients of the exploding terms must vanish. In particular, we only need to calculate the constant term of the asymptotic expansion, which will turn out to be contained (if present) in the term $R (Q_R f)(x)$.
     
    First notice that we have
     \begin{equation*}
     \begin{aligned}
       \int_0^R p_t^{\leq 0}(x, x) \dd t &=R \Pi(x, x) + \sum_{\lambda_j < 0} \left(\int_0^R e^{-t \lambda_j} \dd t \right) \varphi_j(x)^2  \\
       &=R \Pi(x, x)  + L_-^{-1}(x, x) - \sum_{\lambda_j < 0} \frac{e^{-\lambda_j R}}{\lambda_j} \varphi_j(x)^2
       \end{aligned}
     \end{equation*}
     so that we are left to evaluate the limit as $R \rightarrow \infty$ of the term
     \begin{equation} \label{RTerm}
     \begin{aligned}
        -2mR (Q_R f)(x) &+ \bigl((2m-n) f(x) - \delta_f\bigr) \left(R\Pi(x, x) - \sum_{\lambda_j < 0} \frac{e^{-\lambda_j R}}{\lambda_j} \varphi_j(x)^2\right).
      \end{aligned}
     \end{equation}
     We have
   \begin{equation} \label{ExpansionQR}
   \begin{aligned}
     R(Q_Rf)(x) &= \int_0^R \int_M p_{R-u}^+(x, y) f(y) p_u^+(y, x) \dd V_g(y)\dd u \\
     &~~~~~~~~~~+  \int_0^R \int_M p_{R-u}^{\leq0}(x, y) f(y) p_u^+(y, x) \dd V_g(y)\dd u \\
     &~~~~~~~~~~~~~~~~~~~~+\int_0^R \int_M p_{R-u}^+(x, y) f(y) p_u^{\leq 0}(y, x) \dd V_g(y)\dd u \\
     &~~~~~~~~~~~~~~~~~~~~~~~~~~~~~~+  \int_0^R \int_M p_{R-u}^{\leq 0}(x, y) f(y) p_u^{\leq 0}(y, x) \dd V_g(y)\dd u
   \end{aligned}
   \end{equation}
   To calculate the asymptotic expansions of these terms, we furthermore split
   \begin{equation*}
     p_u^{\leq 0}(x, x) = \Pi(x, x) + \sum_{\lambda_j < 0} e^{-u\lambda_j} \varphi_j(x) \otimes \varphi_j(x)^* = \Pi(x, x) + p_u^-(x, x)
   \end{equation*}
   in terms of the eigendecomposition of $L$.
   For the first term of \eqref{ExpansionQR}, we get
   \begin{equation*}
   \begin{aligned}
      \left| \int_0^R \int_M p_{R-u}^+(x, y) f(y) p_u^+(y, x) \dd V_g(y) \dd u\right| &
      \leq \|f\|_{\infty}\int_0^R \int_M p_{R-u}^+(x, y) p_u^+(y, x)\dd V_g(y) \dd u \\
      &= \|f\|_{\infty} R p_R^+(x, x),
   \end{aligned}
   \end{equation*}
   which converges to zero as $R \rightarrow \infty$. The second term of \eqref{ExpansionQR} yields
   \begin{equation*}
   \begin{aligned}
    \int_0^R \int_M p_{R-u}^{\leq0}(x, y) f(y) p_u^+(y, x) \dd V_g(y) \dd u &=  \int_0^R \int_M \Pi(x, y) f(y) p_u^+(y, x)  \dd V_g(y)\dd u\\
    &~~~~+ \int_0^R \int_M p_{R-u}^{-}(x, y) f(y) p_u^+(y, x) \dd V_g(y) \dd u
    \end{aligned}
   \end{equation*}
   Here we have
   \begin{equation*}
   \begin{aligned}
   \lim_{R \rightarrow \infty} \int_0^R \int_M \Pi(x, y) f(y) p_u^+(y, x) \dd V_g(y)\dd u  &= \int_M \Pi(x, y) f(y) L_+^{-1}(y, x) \dd V_g(y)\\
   &= [\Pi f L_+^{-1}](x, x)
   \end{aligned}
   \end{equation*} 
   and
   \begin{equation*}
   \begin{aligned}
      \int_0^R \int_M &p_{R-u}^{-}(x, y) f(y) p_u^+(y, x) \dd V_g(y) \dd u\\ 
      &= \sum_{\lambda_j < 0} \varphi_j(x)  e^{-R\lambda_j} \int_0^R \int_M f(y) \varphi_j(y)^* e^{u\lambda_j} p_{u}^+(y, x) \dd V_g(y) \dd t\\
      &= \sum_{\lambda_j < 0} \varphi_j(x) e^{-R\lambda_j} \int_0^R \int_M (L_g-\lambda_j)^{-1} \bigl\{ f \varphi_j\bigr\}(y)^* \frac{\partial}{\partial u} \bigl\{ e^{u\lambda_j} p_{u}^+(y, x)\bigr\} \dd V_g(y) \dd t\\
      &= \sum_{\lambda_j < 0} \varphi_j(x) \int_M (L_g-\lambda_j)^{-1} \bigl\{ f \varphi_j\bigr\}(y)^*  p_R^+(y, x)\dd V_g(y) \\
      &~~~~~~~~~~~~~~~~~~~~~~~~~~~~~~~~~~~~~~~~~~~-  \sum_{\lambda_j < 0} \varphi_j(x) e^{-R\lambda_j} \Pi_g^+ (L_g-\lambda_j)^{-1} \bigl\{ f \varphi_j\bigr\}(x)^*,
   \end{aligned}
   \end{equation*}
   where $\Pi_g^+$ is the projection on the positive spectral part of $L_g$, i.e.\ $\Pi_g^+\varphi_j = \varphi_j$ if $\lambda_j >0$ and $\Pi_g^+\varphi_j = 0$ otherwise. The first term here again converges to zero as $R\rightarrow \infty$, so that we have the asymptotic expansion
   \begin{equation*}
   \begin{aligned}
   \int_0^R &\int_M p_{R-u}^{\leq0}(x, y) f(y) p_u^+(y, x) \dd V_g(y) \dd u \\
   &~~~~~~~~~~~~~~~~~~\sim~   [\Pi f L_+^{-1}](x, x) - \sum_{\lambda_j < 0} e^{-R\lambda_j} \varphi_j(x) \Pi_g^+ (L-\lambda_j)^{-1} \bigl\{ f \varphi_j\bigr\}(y)^* + o(1).
   \end{aligned}
   \end{equation*} 
   The third term of \eqref{ExpansionQR} just yields the pointwise transpose inside $\mathrm{End}(\V_x)$ of this term. The fourth term of \eqref{ExpansionQR} finally gives
   \begin{equation*}
   \begin{aligned}
   \int_0^R &\int_M p_{R-u}^{\leq 0}(x, y) f(y) p_u^{\leq 0}(y, x) \dd V_g(y)\dd u 
   = R [\Pi f \Pi](x, x) \\
   &+ \sum_{\substack{\lambda_i, \lambda_j \leq 0\\ \lambda_i \neq \lambda_j}} \varphi_i(x)\varphi_j(x)^* \frac{e^{-R\lambda_j} - e^{-R\lambda_i}}{\lambda_i - \lambda_j} (\varphi_i, f \varphi_j)_{L^2}+ \sum_{\lambda_j < 0} \varphi_j(x) \varphi_j(x)^* R e^{-R\lambda_j} (\varphi_j, f\varphi_j)_{L^2}
   \end{aligned}
   \end{equation*} 
   The constant term in $R$ is here given by
   \begin{equation*}
   \begin{aligned}
   \sum_{\lambda_i =0, \lambda_j < 0} \varphi_i(x)\varphi_j(x)^* \frac{1}{\lambda_j} (\varphi_i, f \varphi_j)_{L^2}
   + \sum_{\lambda_i <0, \lambda_j = 0} \varphi_i(x)&\varphi_j(x)^* \frac{1}{\lambda_i} (\varphi_i, f \varphi_j)_{L^2}\\
   &= [\Pi f L_-^{-1}](x, x) + [L_-^{-1}f \Pi ](x, x)
   \end{aligned}
   \end{equation*} 
   Putting all terms together and ignoring all divergent terms (knowing that these must cancel), this shows that in the limit $R\rightarrow \infty$ of \eqref{RTerm}, we get
   \begin{equation*}
   \begin{aligned}
    [\Pi_g f L_+^{-1}](x, x) + [L_+^{-1}f \Pi_g &](x, x) + [\Pi_g f L_-^{-1}](x, x) + [L_-^{-1}f \Pi_g ](x, x) \\
    &=  [\Pi_g f L_g^{-1}](x, x) + [L_g^{-1}f \Pi_g ](x, x) =  2[\Pi_g f L_g^{-1}](x, x),
    \end{aligned}
   \end{equation*}
   where the last step follows since the integral kernel of the anti-symmetric part of a smoothing operator evaluates to zero on the diagonal.
  \end{proof}

\begin{remark}
  Taking the residue of \eqref{LargeFormula} at $s=1$ on both sides, we obtain that
  \begin{equation*}
     \delta_f \bigl\{ \res_{s=1} \zeta_g(s, x)\bigr\} = (2m-n)f(x) \res_{s=1}\zeta_g(s, x),
  \end{equation*}
  which gives the well-known pointwise invariant in even dimensions.
\end{remark}

Let us remark that collecting the terms linear in $R$ in the above proof, we obtain the following result.

\begin{corollary}
The integral kernel of the projection onto the kernel $\Pi_g$ evaluated at the diagonal satisfies the variation formula
\begin{equation*}
  \delta_f \Pi_g(x, x) = (2m-n)f(x) \Pi_g(x, x) - 2m[\Pi_g f \Pi_g](x, x).
\end{equation*}
\end{corollary}

    \appendix

\bibliography{Literatur}

\begin{thebibliography}{CRGS92}

\bibitem[AH05]{AmmannHumbert}
B.~Ammann and E.~Humbert.
\newblock Positive mass theorem for the Yamabe problem on spin manifolds.
\newblock {\em Geom. and Func. Anal.}, 15(567-576), 2005.

\bibitem[BGV04]{bgv}
Nicole Berline, Ezra Getzler, and Michele Vergne.
\newblock {\em Heat Kernels and Dirac Operators}.
\newblock Springer, Berlin, Heidelberg, New York, 2004.

\bibitem[BJ10]{BaumJuhl}
H.~Baum and A.~Juhl.
\newblock {\em Conformal Differential Geometry: Q-curvature and Conformal
  Holonomy}, volume~40 of {\em Oberwolfach Seminars}.
\newblock Birkh\"auser, 2010.

\bibitem[Br86]{BransonOrsted}
T.~Branson and B.~\O rsted.
\newblock Conformal indices of Riemannian manifolds.
\newblock {\em Compositio Mathematica}, 60(3):261--293, 1986.

\bibitem[Br91]{BransonOrsted2}
T.~Branson and B.~\O rsted.
\newblock Conformal geometry and global invariants.
\newblock {\em Differential Geometry and its Applications}, 1:279--308, 1991.

\bibitem[Bra85]{BransonDifferentialOperators}
T.~Branson.
\newblock Differential operators canonically associated to a conformal
  structure.
\newblock {\em Math. Scand.}, 57(2):293--345, 1985.

\bibitem[Bra93]{BransonFunctionalDeterminant}
T.~Branson.
\newblock {\em The Functional Determinant}.
\newblock Lecture Notes Series, 4. Seoul National University, Research
  Institute of Mathematics, Global Analysis Research Center, 1993.

\bibitem[Bre03]{Brendle}
Simon Brendle.
\newblock Global existence and convergence for a higher order flow in conformal
  geometry.
\newblock {\em Annals of Mathematics}, 158:323--343, 2003.

\bibitem[CRGS92]{GJMS1}
L.~J.~Mason C.~R.~Graham, R.~Jenne and G.~A.~J. Sparling.
\newblock Conformally invariant powers of the Laplacian. I. Existence.
\newblock {\em J. London Math. Soc.}, 46(2):557--565, 1992.

\bibitem[CY95]{ChangYangAnnals}
S.Y. Chang and P.~C. Yang.
\newblock Extremal metrics of zeta function determinants on 4-manifolds.
\newblock {\em Annals of Mathematics}, 142:171--212, 1995.

\bibitem[DM08]{DjaldiMacholdi}
Z.~Djaldi and A.~Malcholdi.
\newblock Existence of conformal metrics with constant $Q$-curvature.
\newblock {\em Annals of Mathematics}, 168:813--858, 2008.

\bibitem[FG13]{FeffermanGrahamJuhl}
C.~Fefferman and C.~R. Graham.
\newblock Juhl's formulae for GJMS operators and $Q$-curvatures.
\newblock {\em Journal of the American Mathematical Society}, 26(4):1191--1207,
  2013.

\bibitem[GH04]{GoverHirachi}
A.~Rod Gover and K.~Hirachi.
\newblock Conformally invariant powers of the Laplacian -- a complete
  non-existence theorem.
\newblock {\em Journal of the American Mathematical Society}, 17(2):389--405,
  2004.

\bibitem[Gil95]{gilkey95}
Peter Gilkey.
\newblock {\em Invariance Theory, the Heat Equation, and the Atiyah-Singer
  Index Theorem}.
\newblock CRC Press, Boca Raton, 1995.

\bibitem[Gov06]{Gover}
A.~R. Gover.
\newblock Laplacian operators and $Q$-curvature on conformally Einstein
  manifolds.
\newblock {\em Math. Ann.}, 336(2):311--334, 2006.

\bibitem[Gre71]{GreinerHeatEquation}
P.~Greiner.
\newblock An asymptotic expansion for the heat equation.
\newblock {\em Archive for Rational Mechanics and Analysis}, 41(3):163--218,
  1971.

\bibitem[Hab00]{Habermann}
L.~Habermann.
\newblock {\em Riemannian Metrics of Constant Mass and Moduli Spaces of
  Conformal Structures}.
\newblock Lecture Notes in Mathematics. Springer, Berlin, Heidelberg, New York,
  2000.

\bibitem[HH16a]{HermannHumbert}
A.~Hermann and E.~Humbert.
\newblock About the mass of certain second order elliptic operators.
\newblock {\em Adv. Math.}, 294:596--633, 2016.

\bibitem[HH16b]{HermannHumbertOnPositiveMass}
A.~Hermann and E.~Humbert.
\newblock On the positive mass theorem for closed Riemannian manifolds.
\newblock In Sumio~Yamada Lizhen~Ji, Athanase~Papadopoulos, editor, {\em From
  Riemann to Differential Geometry and Relativity}. Springer Proceedings in
  mathematics and statistics, 2016.

\bibitem[HR09]{HumbertRaulot}
E.~Humbert and S.~Raulot.
\newblock Positive mass theorem for the Paneitz-Branson operator.
\newblock {\em Calculus of Variations}, 36:525--531, 2009.

\bibitem[Juh09]{Juhl1}
A.~Juhl.
\newblock {\em Families of Conformally Covariant Differential Operators,
  Q-Curvature and Holography}, volume 275 of {\em Progress in mathematics}.
\newblock Birkh\"auser, 2009.

\bibitem[Juh13]{Juhl2}
A.~Juhl.
\newblock Explicit formulas for GJMS operators and $Q$-curvatures.
\newblock {\em Geometric and Functional Analysis}, 23(4):1278--1370, 2013.

\bibitem[Juh14]{Juhl3}
A.~Juhl.
\newblock On the recursive structure of Branson's $Q$-curvature.
\newblock {\em Math. Res. Lett}, 21(3):1--13, 2014.

\bibitem[Juh16]{Juhl4}
A.~Juhl.
\newblock Heat kernel expansions, ambient metrics and conformal invariants.
\newblock {\em Advances in Mathematics}, 286:545--682, 2016.

\bibitem[Lan70]{Landkof}
N.~S. Landkof.
\newblock {\em Foundations of Modern Potential Theory}.
\newblock Springer, Berlin, Heidelberg, New York, 1970.

\bibitem[LP87]{LeeParker}
J.~M. Lee and T.~H. Parker.
\newblock The Yamabe problem.
\newblock {\em Bulletin of the American Mathematical Society}, 17(1):37--91,
  1987.

\bibitem[Pan83]{Paneitz}
S.~Paneitz.
\newblock A quartic conformally covariant differential operator for arbitrary
  pseudo-Riemannian manifolds, 1983.

\bibitem[Pon14]{PongeLogarithmicSingularity}
R.~Ponge.
\newblock The logarithmic singularities of the green functions of the conformal
  powers of the Laplacian.
\newblock {\em Contemp. Math.}, 630:247--273, 2014.

\bibitem[PR87]{ParkerRosenberg}
T.~Parker and S.~Rosenberg.
\newblock Invariants of conformal Laplacians.
\newblock {\em J. Differential Geometry}, 25(199-222), 1987.

\bibitem[PT82]{ParkerTaubes}
T.~Parker and C.~H. Taubes.
\newblock On Witten's proof of the positive energy theorem.
\newblock {\em Commun. Math. Phys.}, 84(223-238), 1982.

\bibitem[Sch06]{SchoenVariational}
R.~Schoen.
\newblock Variational theory for the total scalar curvature functional for
  Riemannian metrics and related topics.
\newblock {\em Topics in Calculus of Variations}, 1365:120--154, 2006.

\bibitem[Shu01]{Shubin}
M.~A. Shubin.
\newblock {\em Pseudodifferential Operators and Spectral Theory}.
\newblock Springer, Berlin, Heidelberg, New York, 2001.

\bibitem[Ste70]{EliasStein}
E.~M. Stein.
\newblock {\em Singular Integrals and Differentiability Properties of
  Functions}.
\newblock Princeton University Press, Princeton, 1970.

\bibitem[SY79]{SchoenYauADM}
R.~Schoen and S.~T. Yau.
\newblock On the proof of the positive mass conjecture in general relativity.
\newblock {\em Commun. Math. Phys.}, 65(45-76), 1979.

\bibitem[SY81]{SchoenYauADM2}
R.~Schoen and S.~T. Yau.
\newblock Proof of the positive mass theorem. II.
\newblock {\em Commun. Math. Phys.}, 79:231--260, 1981.

\bibitem[SY88]{SchoenYauConfFlat}
R.~Schoen and S.T. Yau.
\newblock Conformally flat manifolds, Kleinian groups and scalar curvature.
\newblock {\em Inv. Math.}, 92:46--71, 1988.

\bibitem[Tre67]{treves67}
F.~Treves.
\newblock {\em Topological Vector Spaces, Distributions and Kernels}.
\newblock Academic Press Inc., Orlando, San Diego, New York, London, 1967.

\bibitem[Wit81]{WittenMass}
E.~Witten.
\newblock A new proof of the positive energy theorem.
\newblock {\em Commun. Math. Phys.}, 80(381-402), 1981.

\end{thebibliography}

\end{document}